\documentclass[11pt,a4paper]{amsart}

\usepackage{mymacros_english_paper}
\usepackage{mymacros_common}

\begin{document}
\title[A mixed characteristic analogue of the perfection of rings]{A mixed characteristic analogue of the perfection of rings and its almost Cohen-Macaulay property}

\author[R. Ishizuka]{Ryo Ishizuka}
\address{Department of Mathematics, Institute of Science Tokyo, 2-12-1 Ookayama, Meguro, Tokyo 152-8551}
\email{ishizuka.r.ac@m.titech.ac.jp}

\author[K. Shimomoto]{Kazuma Shimomoto}
% \address{Department of Mathematics, College of Humanities and Sciences, Nihon University, Setagaya-ku, Tokyo 156-8550, Japan}
\address{Department of Mathematics, Institute of Science Tokyo, 2-12-1 Ookayama, Meguro, Tokyo 152-8551, Japan}
\address{Department of Mathematics, Institute of Science Tokyo, 2-12-1 Ookayama, Meguro, Tokyo 152-8551, Japan}
\email{shimomotokazuma@gmail.com}

\thanks{2020 {\em Mathematics Subject Classification\/}: 13A35, 13B05, 13B35, 14G45}

\keywords{Almost mathematics, Cohen-Macaulay ring, Frobenius map, Non-archimedean Banach ring, Perfectoid ring, Perfect closure, Uniform completion}

%\subjclass{13}
%\subjclass[2000]{Primary 13-XX}
%\subjclass[2000]{Primary ; Secondary}
%\date{\today \, (\printtime)}
%\date{\today}

\begin{abstract}
Over a complete Noetherian local domain of mixed characteristic with perfect residue field, we construct a perfectoid ring which is similar to an explicit representation of a perfect closure in positive characteristic. Then we demonstrate that this perfectoid ring is almost Cohen-Macaulay in the sense of almost ring theory. The proof of this result uses Andr\'e's flatness lemma along with Riemann's extension theorem. We stress that the idea partially originates from the ``perfectoidization'' in the theory of prismatic cohomology.
\end{abstract}

\maketitle 

\setcounter{tocdepth}{1}
\tableofcontents
\section{Introduction}

All rings are assumed to be commutative with unity. The Cohen-Macaulay property has been a central object of study in the theory of Noetherian rings. Although it is not true that every Noetherian ring is Cohen-Macaulay, it has been recently proved that Noetherian local rings are dominated by ``big'' Cohen-Macaulay algebras (see \cite{andre2018Lemme} and \cite{andre2018Conjecture} for these results), which are not necessarily Noetherian. Therefore, it is inevitable to consider the non-Noetherian avatar of the Cohen-Macaulay property.

The aim of this paper is to investigate the following question: For a given complete Noetherian local domain $(R,\mfrakm,k)$ with its absolute integral closure $R^+$, can one find an almost Cohen-Macaulay $R$-algebra $T$ such that $R \subseteq T \subseteq R^+$ with an additional requirement that $T$ is obtained as the normalization (or much smaller integral extension) after adjoining all $p$-power roots of a fixed set generating the maximal ideal $\mfrakm$ to $R$? Following the method of P. Roberts developed in \cite{roberts2008Root} and \cite{roberts2010Fontaine}, our aim is to give an answer in the mixed characteristic case.

Let us explain our motivation. Let \((R, \mfrakm, k)\) be a complete Noetherian local domain of positive characteristic $p>0$ with perfect residue field \(k\). Then we can form the \emph{perfect closure} $R \to R_{\perf} \defeq \colim_F R$ in a functorial way.
The term perfect closure is often called \emph{perfection} or \emph{direct perfection} in the literature. This is a \emph{perfect} ring in the sense that the Frobenius map $F: R_{\perf} \to R_{\perf}$ is an isomorphism. Although $R_{\perf}$ is not necessarily Noetherian, it is known to be almost Cohen-Macaulay \cite{roberts2007Annihilators} (see \Cref{RobertsSinghSrinivas}).
So the perfect closure is expected to have good properties and we want to carry out a similar construction in mixed characteristic.

Since an analog of perfect property in mixed characteristic is \emph{(integral) perfectoid} introduced in \cite{scholze2012Perfectoida, bhatt2018Integral}, one approach to constructing an analog of perfect closure in mixed characteristic is \emph{perfectoidization} constructed by using prismatic theory \cite{bhatt2022Prismsa}.
Recently, H. Cai, S. Lee, L. Ma, K. Schwede, and K. Tucker used the theory of perfectoidization to construct an appropriate perfectoid algebra over some Noetherian ring of mixed characteristic in \cite{cai2023Perfectoid}. This construction is used to define perfectoid signature and perfectoid Hilbert--Kunz multiplicity. While this construction has many good properties thanks to the almost purity theorem \cite[Theorem 10.9]{bhatt2017Lecture}, it is rather hard to see their intrinsic structure.
To fill in this gap, we want to realize our construction by employing the idea of the perfect closure in positive characteristic commutative algebra.
% The original approach by Bhatt-Scholze in \cite{bhatt2022Prismsa} is based on defining an appropriate Grothendieck topology to compute cohomology.
% A second approach, which is due to Bhatt-Lurie and Drinfeld respectively, reformulates prismatic cohomology in a stacky way (see \cite{BL221}, \cite{BL222} and \cite{Drin20}). None of these approaches are easy to grasp.

Apart from their strategy, we focus on an explicit representation of the perfect closure as follows. Let $x_1,\ldots,x_n$ be any set of generators of the maximal ideal $\mfrakm$. Then one can describe the perfect closure \(R_{\perf}\) of $R$ explicitly in a fixed absolute integral closure \(R^+\):
\begin{equation} \label{perfectclosure1}
    R_{\perf}=\bigcup_{j \geq 0} R[x_1^{1/p^j},\ldots,x_n^{1/p^j}],
\end{equation}
where \(\{x_i^{1/p^j}\}_{j \geq 0}\) is a compatible system of \(p\)-power roots of \(x_i\) in \(R^+\) for each \(i=1,\ldots,n\).
As one already knows that taking the perfect closure is functorial, the right side of $(\ref{perfectclosure1})$ does not depend on the choice of $x_1,\ldots,x_n$.
Although $R_{\perf}$ is only defined in positive characteristic, the right side of $(\ref{perfectclosure1})$ can be constructed in mixed characteristic.
So the above speculation naturally lets us consider the mixed characteristic analogue.

% \begin{problem} \label{Problem}
%     Assume that $(R,\mfrakm,k)$ is a complete Noetherian local domain of mixed characteristic \((0, p)\) with perfect residue field \(k\). Choose any set of generators $x_1,\ldots,x_n$ of $\mfrakm$ and their compatible system of \(p\)-power roots \(\{x_i^{1/p^b}\}_{n \geq 0}\) in an absolute integral closure \(R^+\). Then is the $p$-adic completion of the subring of \(R^+\)
%     \begin{equation} \label{perfectclosure2}
%     \bigcup_{j \geq 0} R[x_1^{1/p^j},\ldots,x_n^{1/p^j}]
%     \end{equation}
%     perfectoid, or almost Cohen-Macaulay?
% \end{problem}

% Apparently, $(\ref{perfectclosure1})$ and $(\ref{perfectclosure2})$ do not show any sort of distinctions. However, while the former exists uniquely, the latter does not.
Our main theorem asserts that the right hand side of $(\ref{perfectclosure1})$ offers a substitute as a mixed characteristic analogue of the perfect closure. Moreover, we emphasize that our construction of a perfectoid ring has the advantage that it can be carried over a complete Noetherian local domain directly.
% Some notations are noted at the end of this section.
% See also Remark \ref{BhattProof}.

\begin{maintheorem} \label{maintheorem}
    Let $(R,\mfrakm,k)$ be a complete Noetherian local domain of mixed characteristic $p>0$ with perfect residue field $k$. Let $p,x_2,\ldots,x_n$ be a system of (not necessarily minimal) generators of the maximal ideal $\mfrakm$ such that $p,x_2,\ldots,x_d$ forms a system of parameters of $R$. Choose compatible systems of $p$-power roots 
    \begin{equation} \label{compaseq1}
        \{p^{1/p^j}\}_{j \ge 0},\{x_2^{1/p^j}\}_{j \ge 0},\ldots,\{x_n^{1/p^j}\}_{j \ge 0}
    \end{equation}
    inside the absolute integral closure $R^+$. Let $\widetilde{R}_{\infty}$ (resp. $C(R_{\infty})$) be the integral closure (resp. $p$-root closure) of
    \begin{equation*}
        R_{\infty}  \defeq  \bigcup_{j \ge 0} R[p^{1/p^j},x_2^{1/p^j},\ldots,x_n^{1/p^j}]
    \end{equation*}
    in $R_{\infty}[1/p]$. Let $\widehat{\widetilde{R}}_{\infty}$, $\widehat{R}_{\infty}$, and $\widehat{C(R_{\infty})}$ be the $p$-adic completion of $\widetilde{R}_{\infty}$, $R_{\infty, \infty}$ and $C(R_{\infty, \infty})$, respectively. Then there exists a nonzero element $g \in \widehat{R}_{\infty}$ together with a compatible system of $p$-power roots $\{g^{1/p^j}\}_{j \ge 0} \subseteq \widehat{R}_{\infty}$ of $g$ such that the following properties hold:
    \begin{enumerate}
    \item The ring map $\widehat{R}_{\infty} \to \widehat{\widetilde{R}}_{\infty}$ is $(p)^{1/p^\infty}$-almost surjective.
    \item $\widehat{\widetilde{R}}_{\infty}$ is a perfectoid domain that is a subring of $\widehat{R^+}$. Moreover, the image of $g$ under the map $\widehat{R}_{\infty} \to \widehat{\widetilde{R}}_{\infty}$ is nonzero.
    \item Let $A \defeq  W(k)[|x_2, \dots, x_d|] \hookrightarrow R$ be a module-finite extension taken by Cohen's structure theorem. If $R$ is a normal domain, $\widehat{\widetilde{R}}_{\infty}$ and $\widehat{C(R_{\infty})}$ are $(pg)^{1/p^\infty}$-almost flat and $(pg)^{1/p^\infty}$-almost faithful $A$-algebras. In particular, this is a \((pg)^{1/p^\infty}\)-almost Cohen-Macaulay \(A\)-algebra.
    % $\widehat{\widetilde{R}}_{\infty}$
    % and $\widehat{C(R_{\infty})}$ are $(pg)^{1/p^\infty}$-almost Cohen-Macaulay algebras with respect to $p,x_2,\ldots,x_d$.
    \item Assume that $R$ is a normal domain. Then for the ring $A$ as in $(3)$, there exists a nonzero element $h \in A$ such that $A[1/h] \to \widetilde{R}_{\infty}[1/h]$ is a filtered colimit of finite \'etale $A[1/h]$-algebras contained in $\widetilde{R}_{\infty}[1/h]$.
    \end{enumerate}
\end{maintheorem}
    
% In particular, we have a (partial) answer to Problem \ref{Problem} as follows:

% \begin{corollary}
%     With notation as in Main Theorem \ref{maintheorem}, the ring $\widehat{\widetilde{R}}_{\infty}$ and \(\widehat{C(R_{\infty, \infty})}\) are \((pg)^{1/p^\infty}\)-almost Cohen-Macaulay \(T\)-algebras.
% \end{corollary}

The construction of $\widehat{\widetilde{R}}_{\infty}$ can be considered as a generalization of what was previously constructed for complete regular local rings as in \cite{shimomoto2018Integral}. We point out that it is not clear if one can find $g$ from $R_{\infty}$ instead of $\widehat{R}_{\infty}$. One might feel that \Cref{maintheorem} is similar to \cite{ma2022Short,cai2023Perfectoid} in spirit. However, our construction is much more similar to the construction of the perfect closure and there are some distinctions in the proof. The construction in \cite{ma2022Short, cai2023Perfectoid} relies directly on \emph{Andr\'e's flatness lemma} and \emph{almost purity theorem} based on the theory of \emph{perfectoidization} introduced in \cite{bhatt2022Prismsa}.

Our proof does not require using the theory of perfectoidization and we use more ``classical'' objects. For example, \emph{Riemann's extension theorem}, together with the following Bhatt's theorem (see \cite[Theorem 2.9.12 at page 120]{bhatt2019Perfectoid}), whose proof uses Andre's flatness lemma (without prismatic theory).

\begin{theorem}[Bhatt] \label{flatness}
    Let $A$ be a perfectoid Tate ring with a closed ideal $J \subseteq A$. Let $B$ be the uniform completion of $A/J$ as a Banach ring (see \cite[Remark 1.5.4 and Corollary 1.5.22]{bhatt2019Perfectoid} and Definition \ref{Seminorm}).
    Then $B$ is a perfectoid Tate ring and the natural ring map $A \to B$ is surjective.
\end{theorem}

On the other hand, it is not entirely irrelevant to perfectoidization.
First, because of the first-named author's work \cite{ishizuka2024Calculationa}, we know that our construction \(\widehat{C(R_{\infty})}\) can be obtained by the perfectoidization \((\widehat{R}_{\infty})_{\perfd}\) of \(\widehat{R}_{\infty}\) and so our main theorem tells us that the perfectoidization of a perfectoid ring is almost Cohen-Macaulay (\Cref{RemPerfdAlmostCM}).
Second, we could show something similar statement using prefectoidization (\Cref{almostsur1}, \Cref{BhattProof} and \Cref{BhattProof2}), but it would differ from our techniques.

It will be clear that the construction of $R_{\infty}$ in Main Theorem \ref{maintheorem} has flexibility in the sense that one can choose $p,x_2,\ldots,x_n$ in a rather arbitrarily manner.
Another construction of almost Cohen-Macaulay algebras was recently given in \cite{nakazato2023Variant}.
Most of the properties that hold in the construction in \cite{nakazato2023Variant} can be checked in $\widetilde{R}_{\infty}$, except that $\widetilde{R}_{\infty}$ is not necessarily integrally closed in its total ring of fractions.

% Recently, H. Cai, S. Lee, L. Ma, K. Schwede, and K. Tucker used the theory of perfectoidization to construct a perfectoid algebra over some Noetherian ring of mixed characteristic in \cite{cai2023Perfectoid}. While this construction has many good properties thanks to the almost purity theorem \cite[Theorem 10.9]{bhatt2017Lecture}, it is rather hard to see their intrinsic structure. To fill in this gap, we want to realize our construction by employing the idea of the perfect closure in positive characteristic commutative algebra. In fact, our construction also satisfies the main properties of the construction appearing in \cite[Theorem 4.0.3]{cai2023Perfectoid}.
% It might be possible to define the perfectoid signature and perfectoid Hilbert-Kunz multiplicity by using our construction $\widehat{\widetilde{R}}_{\infty}$ instead of $R_{\perfd}^{A_{\infty}}$ as in \cite[Definition 4.0.8]{cai2023Perfectoid}.

\subsection*{Notation} \label{Notation}
We use some notation following \cite{nakazato2022Finite} and \cite{nakazato2023Variant}.

\begin{itemize}
    \item Let $A$ be a topological ring. A subset $B$ in $A$ is called \emph{bounded} if for any open neighborhood $U$ of $0 \in A$, there exists some open neighborhood $V$ of $0 \in A$ such that $V \cdot B \subseteq U$. An element $x \in A$ is called \emph{power-bounded} if the subset $\{x^n\}_{n \geq 0} \subseteq A$ is bounded. The symbol $A^\circ$ is the set of all power-bounded elements of $A$.

    \item We call a topological ring $A$ a \emph{Tate ring} if $A$ has an open subring $A_0$ whose induced topology is equal to the $t$-adic topology for some $t \in A_0$ which is a unit in \(A\). The element $t$ is called a \emph{pseudo-uniformizer} of $A$ and then we have $A = A_0[1/t]$. A \emph{map of Tate rings} is a continuous homomorphism of topological rings.
    
    \item Let $A_0$ be a ring and let $t$ be a non-zero-divisor of $A_0$. The \emph{Tate ring associated to $(A_0,(t))$} is the topological ring $A_0[1/t]$ whose topology is the linear topology induced from $\{t^n A_0\}_{t \geq 0}$.
    The completion of the Tate ring $A_0[1/t]$ as a topological ring is isomorphic to $\widehat{A_0}[1/t]$ where $\widehat{A_0}$ is the $t$-adic completion of $t$-adic topological ring $A_0$.
    We call this completion $\widehat{A_0}[1/t]$ as \emph{$t$-completion}.
    
    \item A \emph{Banach ring} is a pair of a complete topological ring \(A\) a norm on \(A\) in the sense of \cite[P.63 Definition 1.5.1]{bhatt2019Perfectoid} and the topology on $A$ induced by the norm is equal to the topology of $A$. A \emph{map of Banach rings} is a continuous homomorphism of topological rings.
    
    \item A ring $R$ is called \emph{$t$-adically Zariskian} for some $t \in R$ if the Jacobson radical of $R$ contains $t$. We use the notion of \emph{$G$-Galois extension} in the sense of \cite[Appendix C]{nakazato2023Variant}.
    
    \item For any $p$-torsion-free ring $R$, the \emph{$p$-root closure} of $R$ in $R[1/p]$ is the set of all $x \in R[1/p]$ such that $x^{p^{n}} \in R$ for some $n \geq 0$. This is a subring of $R[1/p]$ by \cite{roberts2008Root} and we denote this ring by $C(R)$. See also Definition \ref{n-rootclosure} in the more general situation.
    
    \item We use the notion of \emph{perfectoid ring} as defined in \cite[Definition 3.5]{bhatt2018Integral} and of \emph{perfectoid Tate ring} as defined to be the one of Fontaine's sense in \cite{bhatt2018Integral} and \cite{fontaine2013Perfectoides}.
    For any $p$-adically complete ring $A$, we consider a multiplicative monoid map $\sharp : A^\flat  \defeq  \lim_{F} A/pA \to A$ which is defined in \cite[Remark 2.0.9]{bhatt2017Lecture}.
    
    \item The notion of \emph{almost Cohen-Macaulay algebra} is defined in \cite[Definition 5.4]{shimomoto2018Integral} (\Cref{DefAlmostFFlat}). A \emph{regular sequence} \(x_1, \dots, x_r\) of a ring \(A\) is a sequence of elements such that the multiplication map \(A/(x_1, \dots, x_{i-1}) \xrightarrow{\times x_{i+1}} A/(x_1, \dots, x_i)\) is injective for each \(i = 1, \dots, r\) and \(A/(x_1, \dots, x_r)A \neq 0\).
    
    \item We use two types of symbols $\widehat{(-)}$ and $(-)^{\land}$ for ideal-adic completion to eliminate visual confusion.
    
\end{itemize}

\subsection*{Acknowledgement}
We would like to thank Shinnosuke Ishiro, Fumiharu Kato, Kei Nakazato, and Shunsuke Takagi for their valuable conversations and suggestions.

Furthermore, Bhargav Bhatt taught us a simple and different proof of Main Theorem \ref{maintheorem} as detailed in Remark \ref{BhattProof} at the end of Section 6. His suggestion strengthened the relationship between perfectoidization and our result. We highly appreciate his contribution.

Also, we would like to thank the referee for their careful reading and valuable comments, which helped us improve the presentation of this paper.

\section{Preliminary Lemmas}

\begin{lemma} \label{integralExt1}
    Let $A \hookrightarrow B$ be an integral ring extension with an element $t \in A$ which is a non-zero-divisor of $B$. Assume that $A$ is integrally closed in $A[1/t]$. Then $A \hookrightarrow B$ induces an injective map:
    \begin{equation*}
        A/t^nA \hookrightarrow B/t^nB
    \end{equation*}
    for any $n>0$. Furthermore, after taking the $t$-adic completion, $A \hookrightarrow B$ induces an injective map $\widehat{A} \hookrightarrow \widehat{B}$.
\end{lemma}
    
\begin{proof}
    Without loss of generality, we may assume $n=1$. It suffices to prove that $tA=A \cap tB$. As the inclusion $\subseteq$ is obvious, we prove the other inclusion. So let $b \in A \cap tB$. Since $t \in A$, it follows that
    \begin{equation*}
        b/t \in A[1/t]~\text{and}~b/t \in B.
    \end{equation*}
    By the hypotheses of the integrality of $B$ over $A$ and integral closedness of $A$ in $A[1/t]$, we find that $b/t \in A$. In other words, $b \in tA$. Since the inverse limit functor is left exact, we get the second assertion.
\end{proof}
    
Following the paper \cite{anderson1990Root}, we recall the definition of root closure of rings \cite[Proposition 1.1]{anderson1990Root}.
    
\begin{definition} \label{n-rootclosure}
    Fix an integer $n>0$. Let $A \hookrightarrow B$ be a ring extension. Then we say that $A$ is \emph{$n$-root closed in $B$} if we have $x^n \in A$ for $x \in B$, then $x \in A$. When $A$ is not $n$-root closed in $B$, set $C_0(A,B) \defeq A$. For $m \in \mathbb{N}$, let $C_{m+1}(A,B)$ be the sub-\(C_m(A,B)\)-algebra generated by the set of those elements $x \in B$ such that $x^n \in C_m(A,B)$. The directed union $C_\infty(A,B) \defeq \bigcup_{m>0}C_m(A,B)$ is called the \emph{total $n$-root closure} of $A$ in $B$.

    In the case of \(p\)-root closedness, Roberts \cite{roberts2008Root} proved that the total \(p\)-root closure \(C(R)\) of a \(p\)-torsion-free ring \(R\) in \(R[1/p]\) can be obtained by
    \begin{equation*}
        C(R) = \set{x \in R[1/p]}{x^{p^n} \in R~\text{for some}~n \geq 0}.
    \end{equation*}
    We call \(C(R)\) the \emph{\(p\)-root closure} of \(R\).
\end{definition}
    
We will be mostly interested in the case that $n$ is a fixed prime number and $B=A[1/n]$.
    
\begin{lemma} \label{integralExt2}
    Fix $n \in \mathbb{N}$. Let $A$ be a ring with an element $t \in A$, and let $A \hookrightarrow B$ be a ring extension such that $A$ is $n$-root closed in $B$ and there exists an integer $N>0$ such that $t^NB \hookrightarrow A$. Then $tB \hookrightarrow A$. In particular, any Noetherian N-2 domain $A$ and its integral closure $B$ in $A[1/t]$ such that $A$ is $n$-root closed in $B$ satisfies these conditions.
\end{lemma}
    
\begin{proof}
    Without losing generality, we may assume that $t^{n^k}B \hookrightarrow A$ for some $k > 0$. Pick any element $b \in B$. Then we get $(tb)^{n^k}=t^{n^k}b^{n^k} \in A$. Since $A$ is $n$-root closed in $B$, we get $tb \in A$, as required.
\end{proof}
    
\Cref{integralExt2} is useful in describing a ring of integral elements of a certain Tate ring in the sense of almost ring theory. Such a ring includes perfectoid rings or more generally, uniform Tate rings (defined in \Cref{Uniform}) over any basic setup.
    
Let us record the following definitions.

\begin{definition}[{cf. \cite[Proposition 2.3.8]{cai2023Perfectoid}}] \label{DefAlmostFFlat}
    Let \(A\) be a ring with an element \(t \in A\) and let \(B\) be an \(A\)-algebra admitting a compatible sequence of \(p\)-power roots \(\{t^{1/p^j}\}_{j \geq 0}\) of \(t\).
    A \(B\)-module \(N\) is \emph{\((t)^{1/p^\infty}\)-almost zero} if \(t^{1/p^n} N = 0\) for any \(n \geq 0\).
    A \(B\)-algebra \(C\) is \emph{\((t)^{1/p^\infty}\)-almost flat over \(A\)} if the \(B\)-module \(\Tor_i^{A}(C, M)\) is \((t)^{1/p^\infty}\)-almost zero for any \(A\)-module \(M\) and any \(i > 0\).
    A \(B\)-algebra \(C\) is \emph{\((t)^{1/p^\infty}\)-almost faithful over \(A\)} if \(C \otimes_A M\) is not \((t)^{1/p^\infty}\)-almost zero for any non-zero \(A\)-module \(M\).

    Assume that \(A\) is a Noetherian local ring with the unique maximal ideal \(\mfrakm\).
    A \(B\)-module \(N\) is \emph{\((t)^{1/p^\infty}\)-almost Cohen-Macaulay} if \(N/\mfrakm N\) is not \((t)^{1/p^\infty}\)-almost zero and any system of parameters \(x_1, \dots, x_d\) of \(A\) satisfies
    \begin{equation*}
        t^{1/p^n} \cdot \frac{((x_1, \dots, x_i):_N x_{i+1})}{(x_1, \dots, x_i)N} = 0
    \end{equation*}
    for each \(i = 0, \dots, d-1\) and any \(n \geq 0\).
\end{definition}

\begin{proposition}[cf. {\cite[Proposition 1.4]{roberts2007Annihilators}}] \label{RobertsSinghSrinivas}
Let $(R,\mfrakm,k)$ be a complete Noetherian reduced local ring of equal characteristic $p>0$ with perfect residue field \(k\). Let us choose a system of parameters $t_1, \ldots, t_d$ of $R$ and take the perfect closure $R_{\perf}$ of $R$. Let $A \defeq k[|t_1,\ldots,t_d|]$ be an unramified complete regular local ring and let $A \hookrightarrow R$ be a module-finite extension taken by Cohen's structure theorem.
    Then there exists a nonzero element \(x \in A\) such that \(t_1, x\) is a regular sequence of \(A\) and the following conditions hold.
    \begin{enumerate}
        \item \(A_{\perf}[1/t_1x] \hookrightarrow R_{\perf}[1/t_1x]\) is finite \'etale.
        \item $R_{\perf}$ is a $(t_1x)^{1/p^\infty}$-almost flat $A$-algebra, i.e., $\Tor_i^{A}(R_{\perf}, M)$ is $(t_1x)^{1/p^\infty}$-almost zero for every $A$-module $M$ and all $i > 0$, and
        \item $R_{\perf}$ is a $(t_1x)^{1/p^\infty}$-almost faithful $A$-algebra, i.e., $R_{\perf} \otimes_A M$ is not $(t_1x)^{1/p^\infty}$-almost zero for every non-zero $A$-module $M$.
        \item In particular, $R_{\perf}$ is a $(t_1x)^{1/p^\infty}$-almost Cohen-Macaulay $A$-algebra.
\end{enumerate}
\end{proposition}

\begin{proof}
    By \cite[Lemma 6.15 (a)]{hochster1990Tight} (or the second paragraph of the proof of \cite[Theorem 8.4]{dietz2007Big}), there exists an integer \(e \geq 0\) such that \(A^{1/p^e} \hookrightarrow R[A^{1/p^e}]\) is generically (finite) \'etale.
    Then there exists a nonzero non-unit element \(x' \in A^{1/p^e}\) such that \(A^{1/p^e}[1/x'] \hookrightarrow R[A^{1/p^e}][1/x']\) is finite \'etale.
    Since $x' \in A^{1/p^e}$, we have $(x')^{p^e} \in A$ and replacing $x'$ with $(x')^{p^n}$ if necessary, we can assume that $x' \in A$. Then there exists $x \in A$ such that $x' = (t_1)^m x$, where $m \in \setZ_{\geq 0}$ and $x \in A \setminus t_1 A$. Since $t_1 A$ is a prime ideal of \(A\), it follows that $t_1, x$ is a regular sequence in $A$ and \(A^{1/p^e}[1/t_1x] \hookrightarrow R[A^{1/p^e}][1/t_1x]\) is finite \'etale.
    Taking perfection, we can show that
    \begin{equation*}
        A_{\perf}[1/t_1x] \cong (A^{1/p^e}[1/t_1x])_{\perf} \hookrightarrow (R[A^{1/p^e}][1/t_1x])_{\perf} \cong R_{\perf}[1/t_1x]
    \end{equation*}
    is finite \'etale.
    By the almost purity theorem in positive characteristic (for example, \cite[Proposition 4.3.4]{bhatt2017Lecture}), the extension \(A_{\perf} \hookrightarrow R_{\perf}\) is \((t_1x)^{1/p^\infty}\)-almost finite \'etale.
    In particular, \(R_{\perf}\) is \((t_1x)^{1/p^\infty}\)-almost flat and \((t_1x)^{1/p^\infty}\)-almost faithful over \(A_{\perf}\) (see \cite[Remarques 1.8.1 (1)]{andre2018Lemme}).
    Since $A \hookrightarrow A_{\perf}$ is faithfully flat, $R_{\perf}$ is a $(t_1x)^{1/p^\infty}$-almost flat and \((t_1x)^{1/p^\infty}\)-almost faithful $A$-algebra.
    By definition, almost flat and almost faithful property over \(A\) implies almost Cohen-Macaulay property and this completes the proof.
\end{proof}

\section{Completion and Integral Extensions of Perfect Rings}

We examine some perfect(oid) assumptions under which a certin integral extension stays integral after completion. First we record some lemmas.
The folloing one is a similar result to \Cref{integralExt1} but for perfect rings.

\begin{lemma}
\label{AlmostSurjectiveIntClosure}
Let \(A\) be a perfect ring of characteristic $p>0$ and let \(t\) be a non-zero-divisor of \(A\). For any integral extension \(A \hookrightarrow B\) such that $A[1/t]=B[1/t]$, the \(t\)-adic completion \(\widehat{A} \to \widehat{B}\) is (honest) injective and \((t)^{1/p^\infty}\)-almost surjective.
\end{lemma}

\begin{proof}
    This proof proceeds as in the proof of \cite[Proposition 4.3.4]{bhatt2017Lecture}: for any \(f \in B\), the finitely generated \(A\)-submodule in \(A[1/t]\) spanned by \(\{f, f^2, \dots, f^N, \dots\}\) is contained in \(t^{-c}A\) for some \(c > 0\).
    So we have \(t^cf^{p^n} \in A\) for any \(n \geq 0\) and thus \(t^{c/p^n}f \in A\) since \(A\) is perfect. This shows that the inclusion map \(A \hookrightarrow B\) is \((t)^{1/p^\infty}\)-almost isomorphism.
    By \cite[Lemma 2.3]{nakazato2023Variant} and the \(t\)-torson-freeness of \(\widehat{A}\), its \(t\)-adic completion \(\widehat{A} \hookrightarrow \widehat{B}\) is injective and \((t)^{1/p^\infty}\)-almost surjective.
\end{proof}

The following lemma is a stability of regular sequences under the normalization of finite \'etale extensions.

\begin{lemma}
\label{StabilityRegSeq}
Let \(A\) be a normal domain over a field \(K\) and let \(t, x\) be a regular sequence of \(A\). Assume that \(A\) is \(t\)-adically Zariskian. Take an integral extension \(A \hookrightarrow B\) of integral domains. If \(A \hookrightarrow B\) becomes finite \'etale after inverting \(tx\), then \(t, x\) is a regular sequence of \(\widetilde{B}\), where \(\widetilde{B}\) is the integral closure of \(B\) in \(B[1/tx]\).
\end{lemma}

\begin{proof}
    % If \(x\) is a unit of \(A\), then the conclusion is trivial. So we may assume that \(x\) is not a unit of \(A\). 
    % In particular, \(A/(t, x)A\) is a non-zero ring because of the \(t\)-Zariskian property of \(A\).
    Since \(A \hookrightarrow \widetilde{B}\) is an integral extension, there exists a prime ideal of \(\widetilde{B}\) such that its contraction to \(A\) contains \((t, x)A\) because of \(A/(t, x)A \neq 0\). This shows that \(\widetilde{B}/(t, x)\widetilde{B} \neq 0\).
    Since \(A[1/tx]\) is a normal domain and \(A[1/tx] \hookrightarrow B[1/tx]\) is a finite \'etale extension, the localization \(B[1/tx]\) is also normal and thus so is \(\widetilde{B}\).
    
    % We claim that \(K[T, X] \to A\) which sends \(T \mapsto t\) and \(X \mapsto x\) is an injective map: let \(f(T, X) \in K[T, X]\) be a polynomial such that \(f(t, x) = 0\) in \(A\). Write \(f(T, X)\) as
    % \begin{equation*}
    %     f(T, X) = f_n(T)X^n + \cdots + f_1(T)X + f_0(T) \in K[T, X]
    % \end{equation*}
    % where \(f_i(T) \in K[T]\) for each \(i\).
    % Since \(f(t, x) = 0\) in \(A\), we have \(f_n(t)x^n + \cdots + f_1(t)x + f_0(t) = 0\) in \(A\). By taking modulo \((t, x)\), we have an injection \(K \hookrightarrow A/(t, x)A\) and thus \(0 = f(t, x) \equiv f_0(t) \equiv f_0(0)\) modulo \((t, x)A\). This means that \(f_0(T)\) has no constant term and thus we have \(f_n(t)x^{n-1} + \cdots + f_2(t)x + f_1(t) \in tA\) since \(x\) is a non-zero-divisor of \(A/tA\). Similarly, \(f_i(T)\) has no constant term and thus \(f(T, X)\) can be divided by \(T\) in \(K[T, X]\). Since \(t\) is a non-zero-divisor of \(A\), we have \(f(T, X) = Xg(T, X)\) for some \(g(T, X) \in K[T, X]\) such that \(g(t, x) = 0\) in \(A\). By repeating this process, we get \(f(T, X) = 0\) in \(K[T, X]\).
    % So we have shown that the map \(K[T, X] \to A\) is injective.
 
Take a map from the polynomial algebra \(K[T, X]\) to \(A\) by \(T \mapsto t\) and \(X \mapsto x\). Then we can write \(A\) as a filtered colimit of finite-type \(K[T, X]\)-algebras. Since \(A\) is normal, it can be rewritten as a filtered colimit: \(A = \colim_{\lambda} A_{\lambda}\) such that $A_\lambda$ is a finite-type normal \(K[T, X]\)-algebra contained in $A$ (see \citeSta{0335}). First we show that \(t, x\) is a regular sequence on \(A_{\lambda}\) for each \(\lambda\): It suffices to show that for each prime ideal $\mfrakp$ of $A_\lambda$ such that $(t,x) \subseteq \mfrakp$, the sequence $t,x$ is regular on the localization $(A_\lambda)_\mfrakp$ with respect to the multiplicative set $A_\lambda \setminus \mfrakp$. Then we have a ring map $(A_\lambda)_\mfrakp \to A_\mfrakp:=A \otimes_{A_\lambda} (A_{\lambda})_\mfrakp$ and $(A_\lambda)_\mfrakp$ is an excellent normal loca domain. In particular, $(A_{\lambda})_\mfrakp$ is an equi-dimensional local domain.
By Serre's normality criterion, it suffices to see that $t,x$ forms part of a system of parameters of $(A_\lambda)_\mfrakp$. To get a contradiction, suppose that there is a prime ideal $\mfrakq \subseteq (A_{\lambda})_\mfrakp$ that is minimal over $t(A_{\lambda})_\mfrakp$ such that $x \in \mfrakq$ in $(A_{\lambda})_\mfrakp$. Then $\mfrakq$ has height one. Consider the localization map $(A_\lambda)_\mfrakq \to A_\mfrakq:=A \otimes_{A_\lambda} (A_{\lambda})_\mfrakq$. Then $t,x$ remains a regular sequence on $A_\mfrakq$. Then we can find a prime ideal $\mfrakP \subseteq A_\mfrakq$ such that $t \in \mfrakP$, but $x \notin \mfrakP$. Pulling this prime ideal back to $A_\mfrakq$, we get a chain of prime ideals
$$
(0) \subsetneq (A_\lambda)_\mfrakq \cap \mfrakP \subsetneq \mfrakq A_\mfrakq.
$$
However, this is impossible because $(A_\lambda)_\mfrakq$ is a discrete valuation ring. Hence we proved that $t,x$ is a regular sequence on $A_\lambda$ for each $\lambda$.

%any associated prime \(\mfrakp\) of \(tA_{\lambda}\) has height one. Taking the localization \((A_{\lambda})_{\mfrakp} \hookrightarrow A_{\mfrakp}\) by the multiplicative closed subset \(A_{\lambda} \setminus \mfrakp\), we know that \(t\) is non-unit of the source. Then any associated prime \(\mfrakq\) of \(tA_{\mfrakp}\) satisfies \(x \notin \mfrakq\) since \(t, x\) is a regular sequence on \(A\).
%Since \((A_{\lambda})_{\mfrakp}\) is a discrete valuation ring, the contraction \(\mfrakq \cap (A_{\lambda})_{\mfrakp}\) which contains \(t\) should be the maximal ideal \(\mfrakp(A_{\lambda})_{\mfrakp}\). Consequently, the condition \(x \notin \mfrakq\) implies that \(x \notin \mfrakp\) and thus \(t, x\) is a regular sequence on \(A_{\lambda}\) for each \(\lambda\).

Since \(A \hookrightarrow \widetilde{B}\) is an integral extension, we can write \(\widetilde{B}\) as a filtered colimit \(\widetilde{B} = \colim_{\mu, \lambda} \widetilde{B}_{\lambda, \mu}\) of normal module-finite sub-\(A_{\lambda}\)-algebras \(\{\widetilde{B}_{\lambda, \mu}\}_{\mu}\) as above for each \(\lambda\). Any associated prime \(\mfrakq\) of \(t\widetilde{B}_{\lambda, \mu}\) has height one and its contraction \(\mfrakq \cap A_{\lambda}\) is also a prime ideal of height one and contains \(t\) since this inclusion \(A_{\lambda} \hookrightarrow \widetilde{B}_{\lambda, \mu}\) has the going-down property.
    Since \(t, x\) is a regular sequence on \(A_{\lambda}\) for each \(\lambda\), the height one prime ideal \(\mfrakq \cap A_{\lambda}\) does not contain \(x\) and so does \(\mfrakq\). Therefore, \(t, x\) form a regular sequence of \(\widetilde{B}_{\lambda, \mu}\) for any \(\mu\) and \(\lambda\). This completes the proof.
\end{proof}

\begin{theorem}
\label{completionintegral}
Let \(A \to B\) be a ring map between perfect rings in positive characteristic \(p>0\) and let \(t\) and \(x\) be non-zero-divisors of \(A\). Suppose the following conditions.
\begin{enumerate}
\item
\(t, x\) is a regular sequence of \(A\). \label{regularseq}

\item
\(A\) is a normal domain. \label{normal}

\item
\(A\) is \(t\)-adically Zariskian. \label{zariskian}

\item
\(t\) and \(x\) are non-zero-divisors of \(B\). \label{nonzerodivisor}
        % \item \(B\) is integrally closed in \(B[1/t]\). \label{integrallyclosed}
\item
\(A \to B\) is an integral extension. \label{integral}

\item
\(A[1/tx] \to B[1/tx]\) is finite \'etale. \label{finiteetale}
\end{enumerate}
Set \(B^+_{B[1/t]}\) for the integral closure of \(B\) in \(B[1/t]\).
Then the \(t\)-adic completion of \(A \hookrightarrow B \hookrightarrow B^+_{B[1/t]}\) extends to an injection \(\widehat{A} \hookrightarrow \widehat{B^+_{B[1/t]}}\), which factors through \(\widehat{A} \hookrightarrow \widehat{A}^+_{\widehat{B^+_{B[1/t]}}} \hookrightarrow \widehat{B^+_{B[1/t]}}\) and the inclusion \(\widehat{A}^+_{\widehat{B^+_{B[1/t]}}} \hookrightarrow \widehat{B^+_{B[1/t]}}\) is \((x)^{1/p^\infty}\)-almost surjective, where \(\widehat{A}^+_{\widehat{B^+_{B[1/t]}}}\) is the integral closure of \(\widehat{A}\) in \(\widehat{B^+_{B[1/t]}}\). Moreover, \(\widehat{A}^+_{\widehat{B^+_{B[1/t]}}}\) and \(\widehat{B}\) are \((x)^{1/p^\infty}\)-almost isomorphic.
\end{theorem}

\begin{proof}
    Once we have shown the \((x)^{1/p^\infty}\)-almost surjectivity of \(\widehat{A}^+_{\widehat{B^+_{B[1/t]}}} \hookrightarrow \widehat{B^+_{B[1/t]}}\), we can conclude the last assertion from \Cref{AlmostSurjectiveIntClosure}. Moreover, we may assume that \(B\) is integrally closed in \(B[1/t]\) by \Cref{AlmostSurjectiveIntClosure}, namely, \(B = B^+_{B[1/t]}\).

    By \Cref{StabilityRegSeq}, $t, x$ forms a regular sequence on the integral closure $\widetilde{B}$ of $A$ in $B[1/tx]$. Now $B \to \widetilde{B}$ is injective by (\ref{nonzerodivisor}).
    Applying \Cref{integralExt1} for $B \hookrightarrow \widetilde{B}$ and $t \in \widetilde{B}$, we get a commutative diagram
    \begin{center}
    \begin{tikzcd}
    B/tB \arrow[r, hook] \arrow[d, "\times x"] & \widetilde{B}/t\widetilde{B} \arrow[d, "\times x", hook] \\
    B/tB \arrow[r, hook]   & \widetilde{B}/t\widetilde{B}
    \end{tikzcd}
    \end{center}
    % \begin{equation*}
    %     \xymatrix{
    %         B/tB  \ar[d]^-{\times x} \ar@{^{(}->}[r] & \widetilde{B}/t\widetilde{B} \ar[d]^-{\times x} \\
    %         B/tB \ar@{^{(}->}[r] & \widetilde{B}/t\widetilde{B}
    %     }
    % \end{equation*}
    and then $t, x$ also forms a regular sequence on $B$.

    The following proof is taken from \cite[Theorem 5.20]{nakazato2023Variant}. Taking a Galois covering of the finite \'etale map $A[1/tx] \hookrightarrow B[1/tx]$ by \cite[Lemma 9.4]{nakazato2023Variant}, there exists an extension
    \begin{equation*}
    A[1/tx] \hookrightarrow B[1/tx] \hookrightarrow C',
    \end{equation*}
    where $A[1/tx] \hookrightarrow C'$ is a $G$-Galois covering for some finite group $G$. In particular, since $(C')^G = A[1/tx]$, the composite map $A[1/tx] \hookrightarrow C'$ is an integral extension. Let $C$ be the integral closure of $A$ in $C'$. Since $C[1/tx]$ is the integral closure of $A[1/tx]$ in $C'[1/tx] = C'$ and $A[1/tx] \hookrightarrow C'$ is integral, we have $C'=C[1/tx]$. Note that $C$ is integrally closed in $C[1/tx]$. By using \Cref{StabilityRegSeq}, \(t, x\) is also a regular sequence on \(C\).
    Now the following commutative diagram holds:
    \begin{center}
        \begin{tikzcd}
            {A[1/tx]} \arrow[rr, "G\text{-Galois}", hook]  &   & {C[1/tx]} \\
            A \arrow[rd, "\text{integral}"', hook] \arrow[rr, "\text{integral}", hook] \arrow[u, "\text{int.cls}"', hook] &   & C \arrow[u, "\text{int.cls}", hook] \\
            & B \arrow[ru, "\text{integral}"', hook] & 
        \end{tikzcd}
    \end{center}
    % \begin{equation*}
    %     \xymatrix{
    %           & A[1/tx] \ar@{^{(}->}[rr]^-{G\text{-Galois}} &       & C[1/tx]   \\
    %         R \ar[r]^-{\text{integral}} & A \ar@{^{(}->}[rr]^-{\text{integral}} \ar@{^{(}->}[rd]_-{\text{integral}} \ar@{^{(}->}[u]^-{\text{int.cls}}    &       & C \ar@{^{(}->}[u]^-{\text{int.cls}}        \\
    %           &         & B \ar@{^{(}->}[ru]_-{\text{integral}}    &           \\
    %     }
    % \end{equation*}
    where the injection $B \hookrightarrow C$ is induced from the integrality of $A \hookrightarrow B$ and injectivity of $B \hookrightarrow B[1/tx] \hookrightarrow C'=C[1/tx]$. Now since $t$ and $x$ are invertible in $C' = C[1/tx]$, these elements are non-zero-divisors in $C$. Since \(A\) is normal (\ref{normal}) and \(B\) is integrally closed in \(B[1/t]\), we can use \Cref{integralExt1} to the integral extensions $A \hookrightarrow C$ and $B \hookrightarrow C$ and it follows that the $t$-adic completion maps $\widehat{A} \to \widehat{C}$ and $\widehat{B} \to \widehat{C}$ are injective. In particular, this show the first assertion that the $t$-adic completion of \(A \hookrightarrow B\) extends to an injection \(\widehat{A} \hookrightarrow \widehat{B^+_{B[1/t]}}\).

    Since $G$ acts on $\widehat{C} = \varprojlim_{n > 0} C/t^nC$ and $\widehat{A} = \varprojlim_{n > 0} A/t^nA$ in a componentwise manner, we have an injective map $\widehat{A} \hookrightarrow (\widehat{C})^G$. We first show that the inclusion map $\widehat{A} \hookrightarrow (\widehat{C})^G$ is $(x)^{1/p^\infty}$-almost surjective.

    By the hypothesis (\ref{zariskian}) and the integrality of the extension $A \hookrightarrow C$, $C$ is also $t$-adically Zariskian. This shows that $A$ and $C$ satisfy the hypotheses of Zariskian Riemann's extension theorem \cite[Theorem 5.11]{nakazato2023Variant} for the elements $t, x$.
    Let $A^j$ and $C^j$ be the Tate ring associated to $\big(A[t^j/x] , (t)\big)$ and $\big(C[t^j/x] , (t)\big)$, respectively. Then Zariskian Riemann's extension theorem and integral closedness of $A$ and $C$ show that the canonical map \(A^+_{A[1/tx]} \to A^{j \circ}\) which is compatible with \(G\)-action induces an isomorphism
    % \begin{align}
    % A & = A^+_{A[1/tx]} \xrightarrow{\cong} \varprojlim_{j > 0} A^{j \circ} \quad \text{and} \\
    % C & = C^+_{C[1/tx]} \xrightarrow{\cong} \varprojlim_{j > 0} C^{j \circ}.
    % \end{align}
    \begin{equation} \label{ZariskiRiemannExt}
        A = A^+_{A[1/tx]} \xrightarrow{\cong} \varprojlim_{j > 0} A^{j \circ} \quad \text{and} \quad C = C^+_{C[1/tx]} \xrightarrow{\cong} \varprojlim_{j > 0} C^{j \circ}.
    \end{equation}

    By hypotheses of \Cref{completionintegral} (\ref{normal}), $A$ is integrally closed in $A[1/t]$. Since $C$ is the integral closure of \(A\) in $C' = C[1/tx]$, \(C\) is also integrally closed in $C[1/t]$. Furthermore, since $A$ and $C$ are integral over a Noetherian ring $R$, one applies \cite[Proposition 7.1]{nakazato2023Variant} to show that $A$ and $C$ are completely integrally closed in $A[1/t]$ and $C[1/t]$, respectively.
    Note that $A[1/tx] \hookrightarrow C[1/tx]$ is a $G$-Galois covering and \(A[1/tx]\) is perfect, so it is finite \'etale by \cite[Proposition 9.3 (1)]{nakazato2023Variant} and thus $C[1/tx]$ is perfect by \cite[Lemma 4.3.8]{bhatt2017Lecture}. Since $C$ is integrally closed in $C[1/tx]$, $C$ is also perfect. Therefore, we can apply (\ref{CompletionLimit}) in \Cref{AppWitt-perfectRiemannExt} to the perfect rings $A$ and $C$ and the elements $t, x$ to get injective and \((x)^{1/p^\infty}\)-almost surjective maps
    \begin{equation} \label{almostisoring}
        \widehat{\varprojlim_{j > 0} A^{j \circ}}  \hookrightarrow \varprojlim_{j > 0} \widehat{A^{j \circ}} \quad \text{and} \quad
        \widehat{\varprojlim_{j > 0} C^{j \circ}}  \hookrightarrow \varprojlim_{j > 0} \widehat{C^{j \circ}}.
\end{equation}

Then these yield a commutative diagram
    \begin{center}
        \begin{tikzcd}
            \left(\widehat{C}\right)^G \arrow[rr, "\cong"]   &  & \left(\widehat{\varprojlim_{j > 0} C^{j \circ}}\right)^G \arrow[rrr, "(x)^{1/p^\infty} \text{-almost surj}", hook] &  &  & \left(\varprojlim_{j > 0} \widehat{C^{j \circ}}\right)^G \arrow[r, Rightarrow, no head] & \varprojlim_{j > 0} \left(\widehat{C^{j \circ}}\right)^G \\
            \widehat{A} \arrow[u, hook'] \arrow[rr, "\cong"] &  & \widehat{\varprojlim_{j > 0} A^{j \circ}} \arrow[rrrr, "(x)^{1/p^\infty} \text{-almost surj}", hook] \arrow[u]   &  &  &  & \varprojlim_{j > 0} \widehat{A^{j \circ}} \arrow[u, "\cong"']
        \end{tikzcd}
    \end{center}
    where the first and second vertical maps are induced by $A^{j \circ} \hookrightarrow (C^{j \circ})^G$, the right-most vertical isomorphism is deduced from \Cref{GaloisStableAj}, and the middle horizontal maps are injective and $(x)^{1/p^\infty}$-almost surjective by (\ref{almostisoring}) above.
    Chasing the diagram, we can see that the left vertical map $\widehat{A} \hookrightarrow (\widehat{C})^G$ is $(x)^{1/p^\infty}$-almost surjective.
    % For any $\gamma \in (\widehat{C})^G$, let $\gamma'$ be the image of $\gamma$ via $(\widehat{C})^G \hookrightarrow \varprojlim_{j > 0} (\widehat{C^{j \circ}})^G$. By $(x)^{1/p^\infty}$-almost surjectivity of $\widehat{A} \hookrightarrow \varprojlim_{j > 0} (\widehat{A^{j \circ}})^G$, for every $k \geq 0$, there exists some $\alpha_k \in \widehat{A}$ such that
    % \begin{equation*}
    % \widehat{A} \ni \alpha_k \longmapsto x^{1/p^k} \gamma' \in \varprojlim_{j > 0} (\widehat{C^{j \circ}})^G.
    % \end{equation*}
    % Since $(\widehat{C})^G \hookrightarrow \varprojlim_{j > 0} (\widehat{C^{j \circ}})^G$ is injective, $\alpha_k$ maps to $x^{1/p^k} \gamma \in (\widehat{C})^G$ and so $\widehat{A} \hookrightarrow (\widehat{C})^G$ is $(x)^{1/p^\infty}$-almost surjective.
    So we get a commutative diagram
    \begin{center}
        \begin{tikzcd}
            \widehat{A} \arrow[rrrd, hook] \arrow[rrrr, "(x)^{1/p^\infty}\text{-almost surj}", hook] &  &  & & \widehat{C}^G \arrow[rr, "\text{integral}", hook] &  & \widehat{C} \\
            &  &  & \widehat{B}. \arrow[rrru, hook] &   &  & 
        \end{tikzcd}
    \end{center}
    % \begin{equation*}
    %     \xymatrix{
    %         \widehat{A} \ar@{^{(}->}[rrrd] \ar@{^{(}->}[rrrr]^-{(x)^{1/p^\infty}\text{-almost surj}} &  &  & & \widehat{C}^G \ar@{^{(}->}[rr]^-{\text{integral}} &  & \widehat{C} \\
    %         &  &  & \widehat{B}. \ar@{^{(}->}[rrru] &   &  & 
    %     }
    % \end{equation*}
    For any $b \in \widehat{B}$, there exists a set of elements $c_0, \dots, c_{n-1} \in (\widehat{C})^G$ such that
    \begin{equation*}
        b^n + c_{n-1} b^{n-1} + \dots + c_1 b + c_0 = 0~\mbox{in}~\widehat{C}.
    \end{equation*}
    For every $k \geq 0$, multiplying by $x^{n/p^k} \in \widehat{A}$, we have
    \begin{equation*}
    (x^{1/p^k} b)^n + (c_{n-1} x^{1/p^k}) (x^{1/p^k} b)^{n-1} + \dots + (c_1 x^{(n-1)/p^k}) (x^{1/p^k} b) + (c_0 x^{n/p^k}) = 0~\mbox{in}~\widehat{C}.
    \end{equation*}
    By $(x)^{1/p^\infty}$-almost surjectivity of $\widehat{A} \hookrightarrow (\widehat{C})^G$, there exist $a_{n-i}$ in $\widehat{A}$ such that
    $c_{n-i} x^{i/p^k} \in \widehat{C}^G$ is the image of $a_{n-i}$ via $\widehat{A} \hookrightarrow \widehat{C}^G$. Since $\widehat{B} \hookrightarrow \widehat{C}$ is injective, we have
    \begin{equation*}
    (x^{1/p^k} b)^n + a_{n-1} (x^{1/p^k} b) + \dots + a_1 (x^{1/p^k} b) + a_0 = 0~\mbox{in}~\widehat{B}.
    \end{equation*}
    Then $x^{1/p^k} b$ is integral over $\widehat{A}$ for any $k \geq 0$, in particular, $x^{1/p^k} b$ is in $\widehat{A}^+_{\widehat{B}}$, the integral closure of $\widehat{A}$ in $\widehat{B}$. This shows that the inclusion map $\widehat{A}^+_{\widehat{B}} \hookrightarrow \widehat{B}$ is $(x)^{1/p^\infty}$-almost surjective.
\end{proof}

\begin{corollary}
\label{AlmostInjectivePerfectClosure}
    Let \(A \to B\) be a ring map between perfect rings of positive characteristic \(p>0\) and let \(t\) and \(x\) be non-zero-divisors of \(A\).
    Suppose that they satisfy the conditions (\ref{regularseq})--(\ref{finiteetale}) in \Cref{completionintegral}. Set \(\widehat{A}^+_{\widehat{B^+_{B[1/t]}}}\) to be the integral closure of \(\widehat{A}\) in \(\widehat{B^+_{B[1/t]}}\).
    Take a perfect \(B^+_{B[1/t]}\)-algebra \(C\) and take the \(t\)-adic completion \(\widehat{A} \to \widehat{B} \to \widehat{B^+_{B[1/t]}} \to \widehat{C}\).
    If \(\widehat{C}\) is an integral domain and the composite map \(\widehat{A} \to \widehat{C}\) is injective, then the map \(\widehat{B} \to \widehat{C}\) is injective.
\end{corollary}

\begin{proof}
    By \Cref{completionintegral}, our assumption gives rise to the following commutative diagram:
    \begin{center}
        \begin{tikzcd}
            \widehat{A} \arrow[rr, "\text{integral}", hook] \arrow[rrrrrd, hook] &  & \widehat{A}^+_{\widehat{B^+_{B[1/t]}}} \arrow[rrr, "(x)^{1/p^\infty}\text{-almost surj}", hook] &  &  & \widehat{B^+_{B[1/t]}} \arrow[d] \\
                                                                                 &  &                                                                                      &  &  & \widehat{C}          
        \end{tikzcd}
    \end{center}
    Since \(\widehat{C}\) is an integral domain, the kernel of the map \(\widehat{A}^+_{\widehat{B^+_{B[1/t]}}} \to \widehat{C}\) is a prime ideal of \(\widehat{A}^+_{\widehat{B^+_{B[1/t]}}}\) whose contraction to \(\widehat{A}\) is zero.
    This implies that the kernel is zero by integrality of the extension, and hence the map \(\widehat{A}^+_{\widehat{B^+_{B[1/t]}}} \hookrightarrow \widehat{C}\) is injective.
    Since \(\widehat{A}^+_{\widehat{B^+_{B[1/t]}}} \hookrightarrow \widehat{B^+_{B[1/t]}}\) is a \((x)^{1/p^\infty}\)-almost isomorphism, the map \(\widehat{B^+_{B[1/t]}} \to \widehat{C}\) is \((x)^{1/p^\infty}\)-almost injective.
    Now, as in the proof of \Cref{completionintegral}, \(t, x\) is a regular sequence on \(B^+_{B[1/t]}\) and then \(x\) is also a non-zero-divisor of \(\widehat{B^+_{B[1/t]}}\). So the map \(\widehat{B^+_{B[1/t]}} \to \widehat{C}\) is an (honest) injective map. Combining this fact together with \Cref{AlmostSurjectiveIntClosure}, we conclude that \(\widehat{B} \to \widehat{C}\) is injective.
\end{proof}

\begin{lemma}
\label{NoetherianCohenAssumptions}
Let \(R\) be a complete Noetherian local domain of characteristic \(p>0\) with perfect residue field \(k\). Fix a module-finite extension \(A \defeq k[|t_1, \dots, t_d|] \hookrightarrow R\) taken by Cohen's structure theorem where \(d \defeq \dim(R)\).
Then there exists a non-zero element \(x \in A\) such that the inclusion map \(A_{\perf} \hookrightarrow R_{\perf}\) satisfy the conditions (\ref{regularseq})--(\ref{finiteetale}) in \Cref{completionintegral} for this \(t \defeq t_1\) and \(x\). In particular, the conclusion of \Cref{completionintegral} and \Cref{AlmostInjectivePerfectClosure} holds.
\end{lemma}

\begin{proof}
    First, we note that the perfection \(A_{\perf}\) is the filtered colimit of complete regular local rings, and hence the conditions (\ref{normal}) and (\ref{zariskian}) in \Cref{completionintegral} are satisfied.
    By construction, the map \(A_{\perf} \hookrightarrow R_{\perf}\) is integral, and thus (\ref{integral}) holds.
    For (\ref{regularseq}), (\ref{nonzerodivisor}), and (\ref{finiteetale}), it is enough to show that there exists a non-zero element \(x \in A\) such that the map \(A_{\perf}[1/t_1x] \to R_{\perf}[1/t_1x]\) is finite \'etale and \(t_1, x\) is a regular sequence of \(A_{\perf}\).
    However, this has already been shown in \Cref{RobertsSinghSrinivas}.
\end{proof}

\begin{corollary}
\label{InjectivePerfectClosure}
    Let \(R\) be a complete Noetherian local domain of characteristic \(p>0\) with perfect residue field \(k\).
    Fix a module-finite extension \(A \defeq k[|t_1, \dots, t_d|] \hookrightarrow R\) taken by Cohen's structure theorem where \(d \defeq \dim(R)\). Let \(C\) be a perfect \(R\)-algebra domain such that \(C\) is integrally closed in \(C[1/t_1]\).
    Assume that the composite map \(A \hookrightarrow R \to C\) is injective and \(\widehat{C}\) is an integral domain.
    Then the \(t_1\)-adic completion map \(\widehat{R_{\perf}} \to \widehat{C}\) is injective.
\end{corollary}

\begin{proof}
    Since \(C\) is a perfect \(R\)-algebra and integrally closed in \(C[1/t_1]\), \(C\) is a perfect \((R_{\perf})^+_{R_{\perf}[1/t_1]}\)-algebra, where \((R_{\perf})^+_{R_{\perf}[1/t_1]}\) is the integral closure of \(R_{\perf}\) in \(R_{\perf}[1/t_1]\).
    By \Cref{AlmostInjectivePerfectClosure} and \Cref{NoetherianCohenAssumptions}, it is sufficient to show that the map \(\widehat{A_{\perf}} \to \widehat{C}\) is injective. 

    Since \(t_1\) is a prime element of \(A\), the \(p^n\)-th root \(t_1^{1/p^n}\) of \(t_1\) is a prime element of \(A^{1/p^n}\). Set a prime ideal \(\mfrakp_n \defeq t_1^{1/p^n}A^{1/p^n}\) of \(A^{1/p^n}\) for each \(n \geq 0\).
    Since \(\mfrakp_n^{kp^n}\) is a \(\mfrakp_n\)-primary ideal of \(A^{1/p^n}\), the \(kp^n\)-th symbolic power \(\mfrakp_n^{(kp^n)}\) of \(\mfrakp_n\) is equal to the \(kp^n\)-th ordinary power \(\mfrakp_n^{kp^n}\) (see, for example, \cite[Exercise 4.13 (iv)]{atiyah1994Introduction}). So we get \(t_1^kA^{1/p^n}_{\mfrakp_n} \cap A^{1/p^n} = t_1^kA^{1/p^n}\) for any \(n \geq 0\).
    In particular, we have
    \begin{equation*}
        t_1^k(A_{\perf})_{\mfrakp_{\infty}} \cap A_{\perf} = t_1^kA_{\perf}
    \end{equation*}
    where the (height one) prime ideal \(\mfrakp_{\infty}\) of \(A_{\perf}\) is the union of \(\mfrakp_n\) for all \(n \geq 0\).
    This means that the map \(\widehat{A_{\perf}} \to \widehat{(A_{\perf})_{\mfrakp_{\infty}}}\) is injective.

    Since \(\widehat{C}\) is an integral domain, \(t_1\) is not a unit element of \(C\) and there exists a prime ideal \(\mfrakq\) of \(C\) such that \(t_1 \in \mfrakq\).
    In the fraction field of \(C\), we have injective maps \((A_{\perf})_{\mfrakp_{\infty}} \hookrightarrow (A_{\perf})_{\mfrakq \cap A_{\perf}} \hookrightarrow C_{\mfrakq}\).
    Note that \((A_{\perf})_{\mfrakp_{\infty}}\) is a valuation ring and thus the injective map \((A_{\perf})_{\mfrakp_{\infty}} \hookrightarrow C_{\mfrakq}\) is faithfully flat. This implies that
    \begin{equation*}
        t_1^k(A_{\perf})_{\mfrakp_{\infty}} = t_1^kC_{\mfrakq} \cap (A_{\perf})_{\mfrakp_{\infty}}
    \end{equation*}
    for any \(k \geq 0\).
    This means that the map \(\widehat{(A_{\perf})_{\mfrakp_{\infty}}} \to \widehat{C_{\mfrakq}}\) is injective. Notice that the composite map \(\widehat{A_{\perf}} \to \widehat{C} \to \widehat{C_{\mfrakq}}\) is equal to the composite map \(\widehat{A_{\perf}} \to \widehat{(A_{\perf})_{\mfrakp_{\infty}}} \to \widehat{C_{\mfrakq}}\), which is injective. This completes the proof.
\end{proof}

\section{Adjoining \(p\)-power Roots in a Noetherian Local Domain}

\begin{construction} \label{ConstRinfty}
    Let $(R, \mfrakm, k)$ be a complete Noetherian local domain with perfect residue field $k$ in mixed characteristic. Fix an absolute integral closure \(R^+\). Then Cohen's structure theorem says that there exists a surjective map:
    \begin{equation} \label{CohenStr}
        S \defeq W(k)[|t_2,\ldots,t_d,t_{d+1},\ldots,t_n|] \twoheadrightarrow R.
    \end{equation}
    Here, we assume that the images of $p, t_2, \ldots, t_d$ in $R$ form a system of parameters of $R$. Indeed by using the prime avoidance lemma, we can make a choice of a system of parameters of a local ring. Denote their images in $R$ by $p, x_2, \ldots, x_n$. Notice that it is not necessarily true that $p, x_2, \ldots, x_n$ forms a set of the minimal generators of $\mfrakm$. 

    Let us define
    \begin{equation*}
        R_{\infty} \defeq \bigcup_{j \ge 0} R[p^{1/p^j}, x_2^{1/p^j},\ldots,x_n^{1/p^j}],
    \end{equation*}
    where this construction occur inside a fixed absolute integral closure $R^+$.
    Recall that we denote by $C(R_{\infty})$ the $p$-root closure of $R_{\infty}$ in $R_{\infty}[1/p]$, and denote by $\widetilde{R}_{\infty}$ the integral closure of $R_{\infty}$ in $R_{\infty}[1/p]$. Moreover, \(\widehat{S}_{\infty} \to \widehat{R}_{\infty}\) is surjective by \citeSta{0315}
\end{construction}

The goal is to show that the $p$-adic completion of $C(R_{\infty})$ and $\widetilde{R}_{\infty}$ are perfectoid.

\begin{lemma} \label{PropertiesAdjoining}
    Let the notation be as above. Then the following properties hold.
    \begin{enumerate}
        \item The Frobenius map \(R_{\infty}/pR_{\infty} \xrightarrow{F} R_{\infty}/pR_{\infty}\) on \(R_{\infty}\) is surjective.
        \item The \(p\)-adic completion \(\widehat{C(R_{\infty})}\) is a perfectoid ring.
        \item The inclusion map \(C(R_{\infty}) \hookrightarrow \widetilde{R}_{\infty}\) is \((p)^{1/p^\infty}\)-almost surjective. In particular, the \(p\)-adic completion \(\widehat{C(R_{\infty})} \to \widehat{\widetilde{R}}_{\infty}\) is a \((p)^{1/p^\infty}\)-almost isomorphism.
        \item The \(p\)-adic completion \(\widehat{\widetilde{R}}_{\infty}\) is a perfectoid ring.
    \end{enumerate}
\end{lemma}

\begin{proof}
(1): As was done in the case of $R_{\infty}$, we can construct $S_{\infty}$ from the formal power series ring $S = W(k)[|t_2, \ldots, t_d, t_{d+1}, \ldots, t_n|]$:
\begin{equation*}
    S_{\infty} \defeq \bigcup_{j \ge 0} S[p^{1/p^j}, t_2^{1/p^j},\ldots,t_n^{1/p^j}].
\end{equation*}
Then by (\ref{CohenStr}), there exists a surjective ring map $S_{\infty} \twoheadrightarrow R_{\infty}$. As it is clear that the Frobenius map on $S_{\infty}/pS_{\infty}$ is surjective, the same holds on $R_{\infty}/pR_{\infty}$.

(2): We follow the method of \cite[Proposition 2.1.8]{cesnavicius2024Purity}, which we recall for the convenience of readers. So it suffices to prove that the Frobenius map on \(C(R_{\infty})/pC(R_{\infty})\) is surjective.

Put $\varpi \defeq p^{1/p}$ and $R_{\infty}^0 \defeq R_{\infty}$. Suppose that the $R_{\infty}$-algebra $R_{\infty}^k$ has been constructed. Take an element $a \in R_{\infty}^k$ with $a^p \in pR_{\infty}^k$. Let $\{a_n\}_{n \ge 0}$ be a sequence of elements of $R_{\infty}^k$ such that $a_{n+1}^p \equiv a_n \pmod{pR_{\infty}^k} $ and $a_0 \defeq a$. Set
\begin{equation*}
    R_{\infty}^{k+1} \defeq R_{\infty}^k\big[a_n/\varpi^{1/p^n}~\big|~n \ge 0~\mbox{and}~a^p \in pR_{\infty}^k\big] \subseteq R_{\infty}[1/p].
\end{equation*}
Put $R_{\infty}^\infty \defeq \bigcup_{k \ge 0} R_{\infty}^k$. By construction, the Frobenius map on $R_{\infty}^\infty/pR_{\infty}^\infty$ induces an isomorphism: $R_{\infty}^\infty/\varpi R_{\infty}^\infty \cong R_{\infty}^\infty/pR_{\infty}^\infty$. Together with \(R_{\infty}^\infty = C(R_{\infty})\) (see \cite[(2.1.7.1)]{cesnavicius2024Purity}), this implies that the $p$-adic completion of $R_{\infty}^\infty = C(R_{\infty})$ is perfectoid.

(3): This is a variant of \Cref{integralExt2}. Let $b \in \widetilde{R}_{\infty}$ be an arbitrary element. Since $C(R_{\infty}) \hookrightarrow \widetilde{R}_{\infty}$ is an integral extension and coincides after inverting $p$, there exists $k \in \mathbb{N}$ such that $p^{k}b^n \in C(R_{\infty})$ for any $n \in \setZ_{\geq 0}$. Then we get
\begin{equation*}
    p^kb^p, p^kb^{p^2}, \ldots \in C(R_{\infty})
\end{equation*}
and hence $(p^{k/p^n}b)^{p^n} \in C(R_{\infty})$. By the $p$-root closedness, it follows that $p^{k/p^n}b \in C(R_{\infty})$. Since $k$ is fixed and $n$ is arbitrary, the $(p)^{1/p^\infty}$-almost surjectivity is proved.
Finally, by \cite[Lemma 2.3]{nakazato2022Finite}, the $p$-adic completion of this inclusion map is a $(p)^{1/p^\infty}$-isomorphism.

(4): Since \(\widetilde{R}_{\infty}\) is \(p\)-root closed, it suffices to show that the Frobenius map $\map{\Frob}{\widetilde{R}_{\infty}/p\widetilde{R}_{\infty}}{\widetilde{R}_{\infty}/p\widetilde{R}_{\infty}}$ is surjective. In view of the assertions (2) and (3), the Frobenius map on $\widetilde{R}_{\infty}/p\widetilde{R}_{\infty}$ is $(p)^{1/p^\infty}$-almost surjective. Now we prove that the Frobenius map on $\widetilde{R}_{\infty}/p\widetilde{R}_{\infty}$ is (honestly) surjective. So let $y \in \widetilde{R}_{\infty}$. Then we have $\varpi y = a^p + pb$ for some $a,b \in \widetilde{R}_{\infty}$, where \(\varpi = p^{1/p}\) as above. Set
\begin{equation*}
    z \defeq a/\varpi^{1/p} \in \widetilde{R}_{\infty}[1/p].
\end{equation*}
Then it follows that
\begin{equation*}
    z^p = a^p/\varpi = y-\varpi^{p-1} b \in \widetilde{R}_{\infty}.
\end{equation*}
By the integral closedness of $\widetilde{R}_{\infty}$ in $\widetilde{R}_{\infty}[1/p]$, we get $z \in \widetilde{R}_{\infty}$ and $\varpi y = \varpi z^p + pb = \varpi z^p + \varpi^p b$. By dividing this equation out by a non-zero-divisor $\varpi$, we find that $y = z^p + \varpi^{p-1}b$ and $\map{\Frob}{\widetilde{R}_{\infty}/\varpi^{p-1}\widetilde{R}_{\infty}}{\widetilde{R}_{\infty}/\varpi^{p-1}\widetilde{R}_{\infty}}$ is surjective. Since $\varpi^{p-1} \widetilde{R}_{\infty} \subseteq \varpi \widetilde{R}_{\infty}$ as ideals, we see that $\map{\Frob}{\widetilde{R}_{\infty}/\varpi\widetilde{R}_{\infty}}{\widetilde{R}_{\infty}/\varpi \widetilde{R}_{\infty}}$ is surjective. Consider the commutative diagram with exact rows:
\begin{center}
    \begin{tikzcd}
        0 \arrow[r] & {\widetilde{R}_{\infty}/\varpi^{1/p} \widetilde{R}_{\infty}} \arrow[d, "F_1"] \arrow[r, "\varpi^{(p-1)/p}"] & {\widetilde{R}_{\infty}/\varpi \widetilde{R}_{\infty}} \arrow[d, "F_2"] \arrow[r] & {\widetilde{R}_{\infty}/\varpi^{(p-1)/p} \widetilde{R}_{\infty}} \arrow[d, "F_3"] \arrow[r] & 0 \\
        0 \arrow[r] & {\widetilde{R}_{\infty}/\varpi \widetilde{R}_{\infty}} \arrow[r, "\varpi^{p-1}"] & {\widetilde{R}_{\infty}/p \widetilde{R}_{\infty}} \arrow[r]   & {\widetilde{R}_{\infty}/\varpi^{p-1} \widetilde{R}_{\infty}} \arrow[r]  & 0
    \end{tikzcd}
\end{center}
% \begin{equation*}
%     \xymatrix{
%         0 \ar[r] & {\widetilde{R}_{\infty}/\varpi^{1/p} \widetilde{R}_{\infty}} \ar[d]^-{F_1} \ar[rr]^-{\times \varpi^{(p-1)/p}} & & {\widetilde{R}_{\infty}/\varpi \widetilde{R}_{\infty}} \ar[d]^-{F_2} \ar[r] & {\widetilde{R}_{\infty}/\varpi^{(p-1)/p} \widetilde{R}_{\infty}} \ar[d]^-{F_3} \ar[r] & 0 \\
%         0 \ar[r] & {\widetilde{R}_{\infty}/\varpi \widetilde{R}_{\infty}} \ar[rr]^-{\times \varpi^{p-1}} & & {\widetilde{R}_{\infty}/p \widetilde{R}_{\infty}} \ar[r]   & {\widetilde{R}_{\infty}/\varpi^{p-1} \widetilde{R}_{\infty}} \ar[r]  & 0
%     }
% \end{equation*}
where the vertical map is induced by the Frobenius map. As the maps $F_1$ and $F_3$ are surjective, it follows that $F_2$ is also surjective by the five lemma. So we are done.
\end{proof}

\section{Uniform Completion of Tate Rings}

The purpose of this section is to prove \Cref{maintheorem} (1) (see \Cref{almostsur1}). 
This can be proved by using perfectoidization introduced in \cite{bhatt2022Prismsa} and its representation by using \(p\)-root closure \cite{ishizuka2024Calculationa}. However, we can give a proof of \Cref{maintheorem} (1) by using the uniform completion of Tate rings instead of perfectoidization.
In order to state the proof of using uniform completions, we will first review the concepts related to it. If the readers are only interested in the proof by perfectoidization, they may skip to the end of this section.

Let $A = A_0[1/t]$ be a Tate ring, where $A_0 \subseteq A$ is a ring of definition with a pseudo-uniformizer $t \in A_0$.
We recall the \emph{uniform completion} of a Tate ring.

\begin{definition} \label{Uniform}
    A Tate ring $A$ is \emph{uniform} if the set of all power-bounded elements $A^\circ$ is bounded, or equivalently, some ring of integral elements of $A$ is a ring of definition of $A$ by \cite[Lemma 2.13]{nakazato2022Finite}.
\end{definition}

Any Tate ring can be equipped with the structure as a seminormed ring as follows (see \cite[Definition 2.26]{nakazato2022Finite} for more details).

\begin{definition}
\label{Seminorm}
Let \(A \defeq A_0[1/t]\) be a Tate ring.
Fix a real number \(c > 1\). Then we can define a seminorm \(\norm{\cdot}_{A_0, t, c} \colon A \to \setR_{\geq 0}\) by
\begin{equation*}
\norm{f}_{A_0, t, c} \defeq \inf_{m \in \setZ} \set{c^m}{t^m f \in A_0}.
\end{equation*}
The seminorm defines a seminormed ring \((A, \norm{\cdot}_{A_0, t, c})\).
The topology of the seminormed ring \((A, \norm{\cdot}_{A_0, t, c})\) is equal to the topology of the Tate ring \(A = A_0[1/t]\). In particular, the topology induced from the norm \(\norm{\cdot}_{A_0, t, c}\) does not depend on the choices of \(A_0\), \(t\), and \(c\).
So we write the seminorm \(\norm{\cdot}_{A_0, t, c}\) as \(\norm{\cdot}\) for simplicity.
    
The \emph{spectral seminorm} attached to this seminorm \(\norm{\cdot}\) is defined as
\begin{equation*}
\norm{f}_{\mathrm{sp}} \defeq \lim_{n \to \infty} \norm{f^n}^{1/n}.
\end{equation*}
By Fekete's subadditivity lemma, we have \(\norm{f}_{\mathrm{sp}} \leq \norm{f}\).
\end{definition}

We recall the following correspondence of a complete uniform Tate ring and its spectral seminorm (see \cite[\S 2.4]{nakazato2022Finite}).

\begin{lemma} \label{UniformPropeties}
    Let \(A = A_0[1/t]\) be a Tate ring equipped with a structure of a seminormed ring \((A, \norm{\cdot})\) as above.
    Then $A$ is uniform if and only if the norm \(\norm{\cdot}\) on $A$ is equivalent to the spectral seminorm $\norm{\cdot}_{\mathrm{sp}}$ induced from \(\norm{\cdot}\).

    If $A$ is uniform, the set of all power-bounded elements $A^\circ$ of $A$ is identified with the unit disk $A_{\norm{\cdot}_{\mathrm{sp}} \leq 1}$ with respect to the spectral seminorm in $A$.
\end{lemma}

\begin{proof}
The first statement follows from \cite[Lemma 2.29]{nakazato2022Finite}. By the paragraph above \cite[Definition 2.26]{nakazato2022Finite}, for any $f \in A^\circ$, there exists a real number $C > 0$ such that $\norm{f^n} \leq C$ for any $n > 0$ and thus, $\norm{f}_{\mathrm{sp}} \leq \lim_{n \to \infty} C^{1/n} = 1$. Conversely, by the first statement, there exists some $C > 0$ such that $C \cdot \norm{x} \leq \norm{x}_{\mathrm{sp}}$ for any $x \in A$. If $f$ is in $A_{\norm{\cdot}_{\mathrm{sp}} \leq 1}$, we have $\norm{f^n} \leq C^{-1} \norm{f^n}_{\mathrm{sp}} = C^{-1} (\norm{f}_{\mathrm{sp}})^n \leq C^{-1}$ for any $n > 0$. This shows that $f$ is a power-bounded element and this completes the proof.
\end{proof}

See \cite[Exercise 7.2.6]{bhatt2017Lecture} and \cite[Definition 2.8.13]{kedlaya2015Relative} for the following definition. This explanation of uniform completion is based on the first author's work \cite[Definition 3.4 and Definition 3.6]{ishizuka2024Calculationa}.

\begin{definition}[Uniformization and uniform completion of a Tate ring] \label{UniformizationUniformCompletion}
    Let $A = A_0[1/t]$ be a Tate ring.
\begin{enumerate}
\item
The \emph{uniformization} $A^u = (A_0[1/t])^u$ of $A = A_0[1/t]$ is defined as a Tate ring $A^u  \defeq  A_0^+[1/t]$ which has a ring of definition $A_0^+$, where $A_0^+$ is some ring of integral elements of $A$ such that $A_0 \subseteq A_0^+$.
In particular, we can take $A_0^+$ as the set of all power-bounded elements $A^\circ$ in $A$.

\item
The \emph{uniform completion} of $A = A_0[1/t]$ is defined as the completion of $A^u = A_0^+[1/t]$, which is $\widehat{A_0^+}[1/t]$. Denote it by $\widehat{A^u}$.
\end{enumerate}
\end{definition}

By \Cref{integraluniform}, the uniformization changes the topology of a Tate ring and this is a well-defined notion as follows.

\begin{lemma}
\label{integraluniform}
For any Tate ring $A = A_0[1/t]$, the topology of the uniformization $A^u$ does not depend on the choice of $A^+_0$. Moreover, in the definition of the uniform completion, $\widehat{A_0^+}$ is an integrally closed ring of definition of $\widehat{A^u}$ and thus \(\widehat{A^u}\) is a complete uniform Tate ring.
\end{lemma}

\begin{proof}
For the first assertion, let $A_0^\circ$ be the complete integral closure of $A_0$ in $A$ (this is equal to the set of power-bounded elements \(A^\circ\) of $A$ by \cite[Lemma 2.13]{nakazato2022Finite}). Then we have $tA_0^\circ \subseteq A_0^+ \subseteq A_0^\circ \subseteq A$ by \cite[Lemma 2.3]{nakazato2022Finite}. So the topology of $A^u$ is completely determined by $A^\circ = A_0^\circ$. Because of \cite[Corollary 2.8]{nakazato2022Finite}, a ring of definition $\widehat{A_0^+}$ of $\widehat{A^u}$ is integrally closed in $\widehat{A^u}$ and thus, $\widehat{A^u}$ is a uniform complete Tate ring.
\end{proof}

We can take a canonical composite map of Tate rings $i : A \to A^u \to \widehat{A^u}$.
This map has a universal property as follows.

\begin{proposition} \label{UniversalityUniformCompletion}
    Let \(A\) be a Tate ring.
    Then the canonical map \(\map{i}{A}{\widehat{A^u}}\) is the universal map to uniform complete Tate rings.
    That is, every map of Tate rings \(\map{h}{A}{B}\), where \(B\) is a uniform complete Tate ring, uniquely factors through \(i \colon A \to A^{\widehat{u}}\).
\end{proposition}

\begin{proof}
    By restricting the map $h:A \to B$ to $A^\circ$, we have a continuous map $A^\circ \to B^\circ$.
    Inverting \(t\), we can get a map of Tate rings \(A^\circ[1/t] \to B^\circ[1/t]\).
    As topological rings, we have \(B = B^\circ[1/t]\) and \(A^u = A^\circ[1/t]\) since \(B\) is uniform and \(A^\circ\) is a ring of integral elements of \(A\) that contains a ring of definition of \(A\).
    Therefore, we have a unique map of uniform Tate rings \(A^u \to B\) and the \(t\)-completion of this map is the desired map \(A^{\widehat{u}} \to B\).
\end{proof}

\begin{remark}
In \cite{bhatt2019Perfectoid} and \cite{kedlaya2015Relative}, the \emph{uniform completion} of a Banach ring $(A,\norm{\cdot})$ is defined as the completion with respect to the spectral seminorm. Denote it by $(\widehat{A}^{\mathrm{sp}}, \norm{\cdot}_{\mathrm{sp}})$ and this is a uniform Banach ring in the sense of \cite{bhatt2019Perfectoid} because of the power-multiplicativity of this norm \(\norm{\cdot}_{\mathrm{sp}}\). This has the same universality in uniform Banach rings as above by \cite[Definition 2.8.13]{kedlaya2015Relative}.
\end{remark}

\begin{proposition}
\label{uniformcompletion}
For any Tate ring $A = A_0[1/t]$, the uniform completion $\widehat{A^u}$ naturally has a structure of a uniform Banach ring. This uniform Banach ring $(\widehat{A^u}, \norm{\cdot}_{\widehat{A^u}})$ satisfies the universality of the uniform completion as a Banach ring of $(A, \norm{\cdot}_{A_0}  \defeq  \norm{\cdot}_{A_0,(t),c})$. In particular, we have an isomorphism of topological rings $\widehat{A^u} \cong \widehat{A}^{\mathrm{sp}}$.
\end{proposition}

\begin{proof}
    The first assertion follows from \Cref{Seminorm} and \Cref{UniformPropeties}.
    Take any uniform Banach ring $(B, \norm{\cdot}_B)$ and a map of Banach rings $\varphi : (A, \norm{\cdot}_{A_0}) \to (B, \norm{\cdot}_B)$.
    Then there exists some $C > 0$ such that $\norm{\varphi(a)}_{B} \leq C \cdot \norm{a}_{A_0}$ for any $a \in A$ by continuity. By \Cref{UniformPropeties}, we can assume that \(\norm{\cdot}_B\) is power-multiplicative and thus we have $\varphi(A^\circ) \subseteq B^\circ = B_{\norm{\cdot}_B \leq 1}$. Fix a positive real number $r > 0$. For any $a \in A^\circ$, we have the inequalities:
    \begin{equation*}
    \norm{\varphi(t^k a)}_B \leq \norm{\varphi(t)}_B^k \cdot \norm{\varphi(a)}_B \leq C \cdot \norm{t}_{A_0}^k = C \cdot c^{-k} < r
    \end{equation*}
    for some $k > 0$. This shows that the map of rings $A = A^\circ[1/t] \to B$ induces a map of topological rings $A^u \to B$ which extends $\varphi$.
    In particular, this induces a map of Banach rings $(A^u, \norm{\cdot}_{A^\circ}) \to (B, \norm{\cdot}_B)$.

    Taking completion, we have a unique map of Banach rings $(\widehat{A^u}, \norm{\cdot}_{A^\circ}) \to (B, \norm{\cdot}_B)$ and this completes the proof.
\end{proof}

Using these notions, we establish the following result, which is one of the most essential part of the main theorem.

\begin{proposition} \label{almostsur1}
    Let the notation be as in \Cref{ConstRinfty}. Then the natural ring map $\widehat{R}_{\infty} \to \widehat{\widetilde{R}}_{\infty}$ is $(p)^{1/p^\infty}$-almost surjective.
\end{proposition}

First, we give a proof of \Cref{almostsur1} by using the uniform completion of Tate rings.

\begin{proof}[Proof of \Cref{almostsur1} via uniform completion]
    Following the notation as in \Cref{ConstRinfty}, we get a map of complete Tate rings:
    \begin{equation*}
    f \colon \widehat{S}_{\infty}[1/p] \to \widehat{R}_{\infty}[1/p].
    \end{equation*}
    Then we find that \(\widehat{S}_{\infty}\) is a perfectoid Tate ring and \(\widehat{S}_{\infty}\) is a ring of integral elements of $A$, because $S_{\infty}$ is integrally closed in $S_{\infty}[1/p]$ and \cite[Corollary 2.8]{nakazato2022Finite}. By the proof of \Cref{PropertiesAdjoining} and \citeSta{0315}, the restriction map \(\widehat{S}_{\infty} \to \widehat{R}_{\infty}\) is surjective and so is \(f\).
    Therefore, we can write $\widehat{R}_{\infty}[1/p] = (\widehat{S}_{\infty}[1/p])/J$ for some closed ideal $J \subseteq \widehat{S}_{\infty}[1/p]$. Consider a map of complete Tate rings:
    \begin{equation*}
    g \colon \widehat{R}_{\infty}[1/p] \to \widehat{\widetilde{R}}_{\infty}[1/p].
    \end{equation*}
    Here, $\widehat{\widetilde{R}}_{\infty}[1/p]$ is a perfectoid Tate ring by \Cref{ConstRinfty} (4), and $\widehat{\widetilde{R}}_{\infty}$ is a ring of power-bounded elements of $\widehat{\widetilde{R}}_{\infty}[1/p]$: in non-completed case, we have \((\widetilde{R}_{\infty}[1/p])^\circ = \widetilde{R}_{\infty}\) by \cite[Lemma 2.13]{nakazato2022Finite} and \cite[Proposition 7.1]{nakazato2023Variant} and thus \(\widehat{\widetilde{R}}_{\infty}\) is a ring of power-bounded elements of \(\widehat{\widetilde{R}}_{\infty}[1/p]\) by using \cite[Corollary 2.8 and Lemma 2.13]{nakazato2022Finite}. By \Cref{integraluniform}, \(\widehat{\widetilde{R}}_{\infty}[1/p]\) is the complete uniform Tate ring.
    The map $\map{g}{\widehat{R}_{\infty}[1/p]}{\widehat{\widetilde{R}}_{\infty}[1/p]}$ satisfies the universal property of the uniform completion (\Cref{UniversalityUniformCompletion}) and thus \(\widehat{\widetilde{R}}_{\infty}[1/p]\) is the uniform completion of \(\widehat{R}_{\infty}[1/p]\), which is the quotient of a perfectoid Tate ring \(\widehat{S}_{\infty}[1/p]\) by a closed idel \(J\).
    Since $\widehat{R}_{\infty}[1/p]$ is a Tate ring, \Cref{uniformcompletion} guarantees that we can apply \Cref{flatness} to the map \(g\). This shows that \(\map{g}{\widehat{R}_{\infty}[1/p]}{\widehat{\widetilde{R}}_{\infty}[1/p]}\) is surjective.
    Now \cite[Lemma 2.8.5 at page 115]{bhatt2019Perfectoid} gives that 
    \begin{equation*}
        \widehat{S}_{\infty} \xrightarrow{f} \widehat{R}_{\infty} \xrightarrow{g} \widehat{\widetilde{R}}_{\infty}
    \end{equation*}
    is $(p)^{1/p^\infty}$-almost surjective and, in particular, the map \(\widehat{R}_{\infty} \to \widehat{\widetilde{R}}_{\infty}\) is $(p)^{1/p^\infty}$-almost surjective.
\end{proof}

Second, we give a proof of \Cref{almostsur1} by using perfectoidization intorduced in \cite{bhatt2022Prismsa} and its representation by using \(p\)-root closure \cite{ishizuka2024Calculationa}.

\begin{proof}[Proof of \Cref{almostsur1} via perfectoidization]
    The map \(\widehat{R}_{\infty} \to \widehat{\widetilde{R}}_{\infty}\) factors through \(\widehat{C(R_{\infty})}\).
    By \Cref{PropertiesAdjoining} (3), the latter map \(\widehat{C(R_{\infty})} \to \widehat{\widetilde{R}}_{\infty}\) is a $(p)^{1/p^\infty}$-almost surjective map.
    Let us show that the former map \(\widehat{R}_{\infty} \to \widehat{C(R_{\infty})}\) is also $(p)^{1/p^\infty}$-almost surjective.
    By \cite[Corollary 6.2]{ishizuka2024Calculationa}, the map \(\widehat{R}_{\infty} \to \widehat{C(R_{\infty})}\) is equal to the canonical map \(\widehat{R}_{\infty} \to (\widehat{R}_{\infty})_{\perfd}\), where \((\widehat{R}_{\infty})_{\perfd}\) is the perfectoidization of the semiperfectoid ring \(\widehat{R}_{\infty}\), and this is surjective by \cite[Theorem 7.4]{bhatt2022Prismsa}.
\end{proof}

Perhaps not true, but we have not been able to answer the following question.

\begin{question} \label{almostsurinj}
    Is the inclusion map $R_{\infty} \hookrightarrow \widetilde{R}_{\infty}$ also $(p)^{1/p^\infty}$-almost surjective?
\end{question}

An affirmative answer to \Cref{almostsurinj} shows that the pair of a ring and its ideal $(R_{\infty, \infty}, (p))$ is preuniform in the sense of \cite[Definition 2.16]{nakazato2022Finite}.
If so, \(R_{\infty}\) is uniform by \cite[Proposition 7.1]{nakazato2023Variant} and \cite[Lemma 2.13]{nakazato2022Finite} and so is \(\widehat{R}_{\infty}[1/p]\).
By the proof of \Cref{almostsur1}, this gives an isomorphism \(\widehat{R}_{\infty}[1/p] \xrightarrow{\cong} \widehat{\widetilde{R}}_{\infty}[1/p]\).
% We only know that this $(p)^{1/p^\infty}$-almost surjectivity is equivalent to the injectivity of $\widehat{R}_{\infty} \to \widehat{\widetilde{R}}_{\infty}$ because of \cite[Proposition 2.4 and Lemma 5.5]{nakazato2022Finite} and \cite[Lemma 3.15]{nakazato2023Variant}.

The following proposition is a special case of \citeSta{0CNF} and this features topological invariance between $R_{\infty}$ and $C(R_{\infty})$.

\begin{proposition}[{\citeSta{0CNF} (see also \cite[Lemma 3.4]{witaszek2022Keels})}]
    Let $p>0$ be a prime number and let $A$ be a $p$-torsion-free \(p\)-adically Zariskian domain. If $C(A)$ is the $p$-root closure of $A$ in $A[1/p]$, then the natural ring map $A \to C(A)$ is a universal homeomorphism. In particular, there is an isomorphism of \'etale sites $\Spec(A)_\mathrm{\acute{e}t} \cong \Spec(C(A))_\mathrm{\acute{e}t}$.
\end{proposition}

\section{Proof of \Cref{maintheorem}}

In what follows, keep the notation as in \Cref{ConstRinfty}.

\begin{proof}[Proof of \Cref{maintheorem} (1)]
This is \Cref{almostsur1}.
\end{proof}

\begin{proof}[Proof of \Cref{maintheorem} (2)]
Let \(R^+\) be an absolute integral closure of \(R\). By \Cref{PropertiesAdjoining} (4), the \(p\)-adic completion \(\widehat{\widetilde{R}}_{\infty}\) is a perfectoid ring.
    Since the \(p\)-adic completion \(\widehat{R^+}\) of \(R^+\) is an integral domain by \cite{heitmann2022Surprisingly}, it is enough to show that the map \(\widehat{\widetilde{R}}_{\infty} \to \widehat{R^+}\) is injective and this is a consequence of \Cref{integralExt1}. Note that \((\widehat{\widetilde{R}}_{\infty})^\flat\) is also an integral domain (see, for example, \cite{dine2022Topologicala,dine2024Tilting}).
\end{proof}

\begin{proof}[Proof of \Cref{maintheorem} (3)]
    We have a \((p)^{1/p^\infty}\)-almost surjective map of perfectoid rings \(\map{\pi}{\widehat{S}_{\infty}}{\widehat{\widetilde{R}}_{\infty}}\).
    Taking tilts, we get a \((p^\flat)^{1/p^\infty}\)-almost surjective map of perfect rings \(\map{\pi^\flat}{(\widehat{S}_{\infty})^\flat}{(\widehat{\widetilde{R}}_{\infty})^\flat}\).
    An explicit calculation shows that \((\widehat{S}_{\infty})^\flat\) is isomorphic to the \(p^\flat\)-adic completion of the perfection of the formal power series ring \(k[|p^\flat, t_2^\flat, \dots, t_n^\flat|]\):
    \begin{equation*}
        (\widehat{S}_{\infty})^\flat \cong (k[|p^\flat, t_2^\flat, \dots, t_n^\flat|]_{\perf})^{\wedge_{p^\flat}}.
    \end{equation*}
    Then \(\pi^\flat\) induces a sequence of maps of perfect rings
    \begin{equation} \label{AlmostSurjNonComplete}
        k[|p^\flat, t_2^\flat, \dots, t_n^\flat|]_{\perf} \twoheadrightarrow \pi^{\flat}(k[|p^\flat, t_2^\flat, \dots, t_n^\flat|]_{\perf}) \hookrightarrow (\widehat{\widetilde{R}}_{\infty})^\flat
    \end{equation}
    such that the whole composite map is \((p^\flat)^{1/p^\infty}\)-almost surjective after taking the \(p^\flat\)-adic completion.
    For notational simplicity, we set
    \begin{equation*}
        R' \defeq \pi^\flat(k[|p^\flat, t_2^\flat, \dots, t_n^\flat|])
    \end{equation*}
    which is a complete Noetherian local sub-domain of \((\widehat{\widetilde{R}}_{\infty})^\flat\) and the residue field is \(k\). The middle term of the sequence (\ref{AlmostSurjNonComplete}) is equal to (the explicit representation of) the perfect closure \(R'_{\perf}\) of \(R'\).
    Namely, the above sequence of maps (\ref{AlmostSurjNonComplete}) gives a \((p^\flat)^{1/p^\infty}\)-almost surjective map of perfect rings
    \begin{equation} \label{AlmostSurjComplete}
        (R'_{\perf})^{\wedge_{p^\flat}} \xrightarrow{(p^\flat)^{1/p^\infty}\text{-almost surj}} (\widehat{\widetilde{R}}_{\infty})^\flat.
    \end{equation}
    We next show that this map is injective by using \Cref{InjectivePerfectClosure}.
    By \Cref{ChoiceSOP}, we can take a module-finite extension:
    \begin{equation}
        A' \defeq k[|p^\flat, t_2^\flat, \dots, t_d^\flat, t_{d+1}^\flat, \dots, t_{d+e}^\flat|] \hookrightarrow R'
    \end{equation}
    in the perfect domain \((\widehat{\widetilde{R}}_{\infty})^\flat\) for \(d = \dim(R)\) and \(e \geq 0\).
    Since \(\widetilde{R}_{\infty}\) is completely integrally closed in \(\widetilde{R}_{\infty}[1/p]\) by \cite[Proposition 7.1]{nakazato2023Variant}, so is the \(p\)-adic completion \(\widehat{\widetilde{R}}_{\infty}\) in \(\widehat{\widetilde{R}}_{\infty}[1/p]\) by \cite[Corollary 2.8]{nakazato2022Finite}.
    Then \cite[Main Theorem 1 (1)]{eto2024Ringtheoretic} shows that the tilt \((\widehat{\widetilde{R}}_{\infty})^\flat\) is also (completely) integrally closed in \((\widehat{\widetilde{R}}_{\infty})^\flat[1/p^\flat]\).
    In conjunction with the fact that \(p^\flat\) is a prime element of \(A'\), the composition map \(A' \hookrightarrow R' \hookrightarrow (\widehat{\widetilde{R}}_{\infty})^\flat\) satisfies the hypotheses of \Cref{InjectivePerfectClosure}. This implies that the map \((R'_{\perf})^{\wedge_{p^\flat}} \to (\widehat{\widetilde{R}}_{\infty})^\flat\) is injective. Consequently, with (\ref{AlmostSurjComplete}), the map \((R'_{\perf})^{\wedge_{p^\flat}} \hookrightarrow (\widehat{\widetilde{R}}_{\infty})^\flat\) is injective and \((p^\flat)^{1/p^\infty}\)-almost surjective and thus the mod \(p^\flat\) map
    \begin{equation} \label{AlmostIsomorphism}
        R'_{\perf}/p^\flat R'_{\perf} \xrightarrow{(p^\flat)^{1/p^\infty}\text{-almost isom}} (\widehat{\widetilde{R}}_{\infty})^\flat/p^\flat (\widehat{\widetilde{R}}_{\infty})^\flat
    \end{equation}
    is a \((p^\flat)^{1/p^\infty}\)-almost isomorphism.

    By \Cref{RobertsSinghSrinivas}, there exists a non-zero element \(x' \in A'\) such that the perfection \(R'_{\perf}\) is \((p^\flat x')^{1/p^\infty}\)-almost flat and \((p^\flat x')^{1/p^\infty}\)-almost faithful over \(A'\). In particular, the \((p^\flat x')^{1/p^\infty}\)-almost isomorphism (\ref{AlmostIsomorphism}) implies that the map
    \begin{equation} \label{AlmostFlatFaithful}
        k[|t_2^\flat, \dots, t_d^\flat|] \to (\widehat{\widetilde{R}}_{\infty})^\flat/p^\flat (\widehat{\widetilde{R}}_{\infty})^\flat
    \end{equation}
    induced from \(k[|p^\flat, t_2^\flat, \dots, t_d^\flat|] \hookrightarrow A' \hookrightarrow R' \hookrightarrow (\widehat{\widetilde{R}}_{\infty})^\flat\) is \((p^\flat x')^{1/p^\infty}\)-almost flat and \((p^\flat x')^{1/p^\infty}\)-almost faithful.
    Set \(x \defeq (x')^\sharp \in \widehat{S}_{\infty}\).
    Passing (\ref{AlmostFlatFaithful}) to the mixed characteristic by isomorphism \((\widehat{\widetilde{R}}_{\infty})^\flat/p^\flat (\widehat{\widetilde{R}}_{\infty})^\flat\), we get a \((px)^{1/p^\infty}\)-almost flat and \((px)^{1/p^\infty}\)-almost faithful map
    \begin{equation}
        W(k)[|t_2, \dots, t_d|]/pW(k)[|t_2, \dots, t_d|] \to \widehat{\widetilde{R}}_{\infty}/p\widehat{\widetilde{R}}_{\infty}
    \end{equation}
    which is induced from \(W(k)[|t_2, \dots, t_d|] \hookrightarrow S \twoheadrightarrow R \to \widehat{\widetilde{R}}_{\infty}\).

    We show that the map \(W(k)[|t_2, \dots, t_d|] \to \widehat{\widetilde{R}}_{\infty}\) is \((px)^{1/p^\infty}\)-almost flat and \((px)^{1/p^\infty}\)-almost faithful.
    Chosse a \(W(k)[|t_2, \dots, t_d|]\)-module \(M\) such that \(p^n M=0\). By induction on \(n\) and considering an exact sequence \(0 \to pM \to M \to M/pM \to 0\), we can show that the map \(W(k)[|t_2, \dots, t_d|] \to \widehat{\widetilde{R}}_{\infty}\) is \(p\)-complete \((px)^{1/p^\infty}\)-almost flat and \(p\)-complete \((px)^{1/p^\infty}\)-almost faithful (\Cref{DefAlmostFFlat}). As in the proof of \cite[Theorem 4.0.3]{cai2023Perfectoid}, it follows that \(\widehat{\widetilde{R}}_{\infty}\) is a \((px)^{1/p^\infty}\)-almost flat and \((px)^{1/p^\infty}\)-almost faithful \(W(k)[|t_2, \dots, t_d|]\)-algebra.
\end{proof}

\begin{lemma} \label{ChoiceSOP}
Let $p^\flat, t_2^\flat, \ldots, t_n^\flat \in R' \subseteq (\widehat{\widetilde{R}}_{\infty})^\flat$ be a sequence of elements defined by using $(\ref{AlmostSurjNonComplete})$. Then there exists a subsequence $t_{j_{d+1}}^\flat, \dots, t_{j_e}^\flat$ of $t_{d+1}^\flat, \dots, t_n^\flat$ such that the extended sequence of elements $p^\flat, t_2^\flat, \dots, t_d^\flat, t_{j_{d+1}}^\flat, \dots, t_{j_e}^\flat$ forms a system of parameters of $R'$, where $e = \dim(R')$.
\end{lemma}

\begin{proof}
    We note that the natural map $R' \hookrightarrow (\widehat{\widetilde{R}}_{\infty})^\flat$ induces an integral extension
    \begin{equation*}
        \overline{R'} \defeq  R'/(p^\flat (\widehat{\widetilde{R}}_{\infty})^\flat \cap R') \hookrightarrow (\widehat{\widetilde{R}}_{\infty})^\flat/p^\flat (\widehat{\widetilde{R}}_{\infty})^\flat \cong \widetilde{R}_{\infty}/p\widetilde{R}_{\infty}
    \end{equation*}
    and $\overline{R'}$ is a complete Noetherian local ring.
    
    Since $\widetilde{R}_{\infty}$ is an integral extension domain over a complete Noether local domain $R$, it follows that $\widetilde{R}_{\infty}$ is Henselian, and hence $\widetilde{R}_{\infty}/p\widetilde{R}_{\infty}$ has a unique maximal ideal. Set $\mfrakn \defeq  \sqrt{(t_2^\flat, \dots, t_d^\flat)\widetilde{R}_{\infty}/p\widetilde{R}_{\infty}}$. We claim that $\mfrakn$ is the maximal ideal of $\widetilde{R}_{\infty}/p\widetilde{R}_{\infty}$. Since $\widetilde{R}_{\infty}$ is an integral extension over a normal domain $R$, it follows from Lemma \ref{integralExt1} that $p\widetilde{R}_{\infty} \cap R=pR$ and thus $R/pR \hookrightarrow \widetilde{R}_{\infty}/p\widetilde{R}_{\infty}$ is an integral extension.
    The contraction \(\mfrakn \cap R/pR\) is a radical ideal which contains \(p, t_2, \dots, t_d\) and hence \(\mfrakn \cap R/pR = \sqrt{(t_2, \dots, t_d)R/pR} = \mfrakm R/pR\).
    The induced map $k \cong R/\mfrakm \hookrightarrow (\widetilde{R}_{\infty}/p\widetilde{R}_{\infty})/\mfrakn$ is an integral extension and thus \(\mfrakn\) is a maximal ideal of \(\widetilde{R}_{\infty}/p\widetilde{R}_{\infty}\).
    % Then since $R/(p,t_2,\ldots,t_d)R$ is an Artinian ring, it follows that $(\widetilde{R}_{\infty}/p\widetilde{R}_{\infty})/\mfrakn$ is zero-dimensional and every element in the maximal ideal of $(\widetilde{R}_{\infty}/p\widetilde{R}_{\infty})/\mfrakn$ is nilpotent (indeed, one can present $(\widetilde{R}_{\infty}/p\widetilde{R}_{\infty})/\mfrakn$ as a direct limit of Artinian rings that are module-finite over $R/(p,t_2,\ldots,t_d)R$). Since this ring is reduced, we conclude that $\mfrakn$ is a maximal ideal of $\widetilde{R}_{\infty}/p\widetilde{R}_{\infty}$.
    
    Since $p,t_2,\ldots,t_d$ is a system of parameters of $R$, and $R/pR \hookrightarrow \widetilde{R}_{\infty}/p\widetilde{R}_{\infty}$ and $\overline{R'} \hookrightarrow \widetilde{R}_{\infty}/p\widetilde{R}_{\infty}$ are integral extensions, it follows from \cite[Exercise 9.8]{matsumura1986Commutative} that
    \begin{align*}
        \dim \overline{R'} & = \dim \widetilde{R}_{\infty}/p\widetilde{R}_{\infty} = \dim R/pR = d-1 \quad \text{and} \\
        d-1 & = \height((t_2^\flat, \dots, t_d^\flat) \widetilde{R}_{\infty}/p\widetilde{R}_{\infty}) \leq \height((t_2^\flat, \dots, t_d^\flat) \overline{R'})
    \end{align*}
    and then $t_2^\flat, \dots, t_d^\flat$ forms a system of parameters of $\overline{R'} = R'/(p^\flat (\widehat{\widetilde{R}}_{\infty})^\flat \cap R')$. We know that $R'$ is a complete local domain, hence it is a catenary domain. This gives that
    \begin{equation*}
    \dim R'=\height(p^\flat (\widehat{\widetilde{R}}_{\infty})^\flat \cap R')+d-1.
    \end{equation*}
    Since \(p^\flat\) is a non-zero-divisor of \(R'\), the height of the ideal \(p^\flat (\widehat{\widetilde{R}}_{\infty})^\flat \cap R'\) is greater than or equal to \(1\).
    Using this fact, the dimension of \(R'/(p^\flat, t_2^\flat, \dots, t_d^\flat)R'\) is equal to \(\height(p^\flat (\widehat{\widetilde{R}}_{\infty})^\flat \cap R') - 1\).
    Since $t_{d+1}^\flat, \dots, t_n^\flat$ is a system of generators of the maximal ideal of the quotient ring, we can find a subsequence $t_{j_{d+1}}^\flat, \dots, t_{j_{e}}^\flat$ of $t_{d+1}^\flat, \dots, t_n^\flat$ such that $p^\flat,t_2^\flat, \dots, t_d^\flat, t_{j_{d+1}}^\flat, \dots, t_{j_e}^\flat$ is a system of parameters of $R'$ where \(e = \dim(R')\). This completes the proof of the lemma.
\end{proof}

\begin{remark} \label{BhattProof}
    % After releasing this paper, Bhargav Bhatt gave us a highly advanced but much simpler proof of Main Theorem \ref{maintheorem}. 
After releasing this paper, Bhargav Bhatt and an anonymous referee informed us of a similar construction that yields the same results as ours. Their proof is very simple and does not require the normality of \(R\). They essentially use the almost purity theorem as in \cite[Theorem 10.9]{bhatt2022Prismsa} and Andr\'e's flatness lemma as in \cite[Theorem 7.14]{bhatt2022Prismsa}. It turns out that there is a strong connection between perfectoidization and our mixed characteristic analogue of perfection.
    Here is the outline of his proof.
    
    Let \(R\) be a complete Noetherian local domain of mixed characteristic \((0,p)\) with perfect residue field \(k\).
    Consider the following commutative diagram:
    \begin{center}
        \begin{tikzcd}[column sep=tiny]
            {A = W(k)[|t_2, \dots, t_d|]} \arrow[r, hook] \arrow[d, hook]    & {S = W(k)[|t_2, \dots, t_n|]} \arrow[r, two heads] \arrow[d]  & R \arrow[d]   \\
            {\widehat{A}_\infty = A[p^{1/p^\infty}, t_2^{1/p^\infty}, \dots, t_d^{1/p^\infty}]^{\wedge_p}} \arrow[r] & {S \widehat{\otimes}_A \widehat{A}_\infty} \arrow[r, two heads] \arrow[d]   & {R \widehat{\otimes}_A \widehat{A}_\infty} \arrow[d] \\
            & {\widehat{S}_\infty = S[p^{1/p^\infty}, t_2^{1/p^\infty}, \dots, t_n^{1/p^\infty}]^{\wedge_p}} \arrow[r, two heads] & {R \widehat{\otimes}_S \widehat{S}_\infty}
        \end{tikzcd}
    \end{center}
    where all squares are pushout squares and the symbol \((-)^{\wedge_p}\) is the (classical) \(p\)-adic completion.
    The claim is the following.
\end{remark}

\begin{proposition} \label{BhattProof2}
    Keep the notation above.
    Let \(g \in A\) be a non-zero element such that \(A[1/g] \to R[1/g]\) is finite \'etale.
    Then the map \(A \to C(R \widehat{\otimes}_S \widehat{S}_\infty)^{\wedge_p}\) is \((g)_{\perfd}\)-almost faithful and \((g)_{\perfd}\)-almost flat, where \(C(R \otimes_S S_{\infty})^{\wedge_p}\) is the \(p\)-adic completion of the \(p\)-root closure of \(R \otimes_S S_{\infty}\).
\end{proposition}

We do not know whether \(C(R \otimes_S S_{\infty})\) and \(C(R_{\infty})\) are isomorphic or not.

\begin{proof}
    Since \(\widehat{A}_\infty\) is a perfectoid ring, we can apply the almost purity theorem from \cite[Theorem 10.9]{bhatt2022Prismsa} for the finite map \(\widehat{A}_\infty \to R \otimes_A \widehat{A}_\infty\).
    Then the perfectoidization \((R \otimes_A \widehat{A}_\infty)_{\perfd}\) is an ordinary ring and the map \(\widehat{A}_\infty \to (R \otimes_A \widehat{A}_\infty)\) is \((g)_{\perfd}\)-almost faithful and \((g)_{\perfd}\)-almost flat (see \cite[Theorem 4.0.3]{cai2023Perfectoid}).
    In particular, \(A \to (R \otimes_A \widehat{A}_\infty)_{\perfd}\) is \((g)_{\perfd}\)-almost faithful and \((g)_{\perfd}\)-almost flat.

Write \(R = S/Q\) for some ideal \(Q \subseteq S\).
    By Andr\'e's flatness lemma (\cite[Theorem 7.14]{bhatt2022Prismsa}), we can show that
    \begin{equation*}
        (R \widehat{\otimes}_A \widehat{A}_\infty)_{\perfd} \to \parenlr{\frac{(R \widehat{\otimes}_A \widehat{A}_{\infty})_{\perfd}[X_{d+1}^{1/p^\infty}, \dots, X_n^{1/p^\infty}]^{\wedge_p}}{(X_{d+1} - (t_{d+1} \otimes 1), \dots, X_n - (t_n \otimes 1))}}_{\perfd}
    \end{equation*}
    is \(p\)-completely faithfully flat. Here if \(n = d+1\), we have
    \begin{align*}
        & \parenlr{\frac{(R \widehat{\otimes}_A \widehat{A}_{\infty})_{\perfd}[X_{d+1}^{1/p^\infty}]^{\wedge_p}}{(X_{d+1} - (t_{d+1} \otimes 1))}}_{\perfd} \cong \parenlr{\frac{((S/Q) \otimes_A A_{\infty})[X_{d+1}^{1/p^\infty}]^{\wedge_p}}{(X_{d+1} - (t_{d+1} \otimes 1))}}_{\perfd} \\
        & \cong  \parenlr{\frac{(S \otimes_A A_{\infty})[X_{d+1}^{1/p^\infty}]^{\wedge_p}}{(X_{d+1} - t_{d+1}, Q)}}_{\perfd} \cong \parenlr{\frac{(S[t_1^{1/p^\infty}, \dots, t_d^{1/p^\infty}, t_{d+1}^{1/p^\infty}])^{\wedge_p}}{(Q)}}_{\perfd} \\
        & \cong ((S_{\infty}/QS_{\infty})^{\wedge_p})_{\perfd} \cong (R \widehat{\otimes}_S \widehat{S}_{\infty})_{\perfd}
    \end{align*}
    and this holds for any \(n\) by the same argument.
    Consequently, the canonical map
    \begin{equation} \label{pCompFlatAndreFlatness}
        R_{\perfd}^{A_{\infty}} = (R \widehat{\otimes}_A \widehat{A}_{\infty})_{\perfd} \to (R \widehat{\otimes}_S \widehat{S}_{\infty})_{\perfd}
    \end{equation}
    is \(p\)-completely faithfully flat and thus the map \(A \to (R \widehat{\otimes}_S \widehat{S}_{\infty})_{\perfd}\) is \(p\)-complete \((g)_{\perfd}\)-almost faithful and \(p\)-complete \((g)_{\perfd}\)-almost flat. As in the proof of \cite[Theorem 4.0.3]{cai2023Perfectoid}, this map becomes \((g)_{\perfd}\)-almost faithful and \((g)_{\perfd}\)-almost flat.

    Since \(R\) is \(p\)-torsion-free, so is the base change \(R \otimes_S S_{\infty}\).
    Its \(p\)-adic completion is a semiperfectoid ring and thus \((R \widehat{\otimes}_S \widehat{S}_{\infty})_{\perfd}\) is isomorphic to \(C(R \otimes_S S_{\infty})^{\wedge_p}\) by \cite[Corollary 5.9]{ishizuka2024Calculationa}.

    % On the other hand, our construction \(\widehat{\widetilde{R}}_{\infty, \infty}\) is \((p)^{1/p^\infty}\)-almost isomorphic to the perfectoidization \((\widehat{R}_{\infty, \infty})_{\perfd}\) of \(R_{\infty, \infty}\) (This also follows by Lemma \ref{preparation1} and \cite[Corollary 6.2]{Ish23}).
    % By applying \((-)_{\perfd}\) to the bottom right square, we have a pushout square
    % \begin{center}
    %     \begin{tikzcd}
    %         {(S \widehat{\otimes}_T \widehat{A}_{\infty, \infty})_{\perfd}} \arrow[r] \arrow[d]              & {(R \widehat{\otimes}_T \widehat{A}_{\infty, \infty})_{\perfd}} \arrow[d]                                                   &  &  &                                            \\
    %         {\widehat{S}_{\infty, \infty} \cong (\widehat{S}_{\infty, \infty})_{\perfd}} \arrow[r] & {(\widehat{R}_{\infty, \infty})_{\perfd}} \arrow[rrr, "(p)^{1/p^\infty}\text{-almost isom}", Rightarrow, no head] &  &  & {\widehat{\widetilde{R}}_{\infty, \infty}}
    %     \end{tikzcd}
    % \end{center}
\end{proof}

\begin{remark} \label{RemPerfdAlmostCM}
    As explained in the second proof of \Cref{almostsur1}, our construction of \(\widehat{C(R_{\infty})}\) is isomorphic to the perfectoidization \((\widehat{R}_{\infty})_{\perfd}\) of \(\widehat{R}_{\infty}\).
    So our theorem (\Cref{maintheorem}) asserts that there exists \(0 \neq \{g^{1/p^j}\}_{j \geq 0} \subseteq \widehat{R}_{\infty}\) such that \((\widehat{R}_{\infty})_{\perfd}\) is \((pg)^{1/p^\infty}\)-almost faithful and \((pg)^{1/p^\infty}\)-almost flat \(A\)-algebra.
    On the other hand, in \cite{cai2023Perfectoid}, they show that the perfectoidization \(R_{\perfd}^{A_{\infty}} \defeq (A_{\infty} \otimes_A R)_{\perfd}\) is a \(g'\)-almost faithful and \(g'\)-almost flat \(A\)-algebra for some non-zero element \(g' \in A\) by using the almost purity theorem \cite[Theorem 10.9]{bhatt2022Prismsa}.
    We do not known of good properties of the map \(R_{\perfd}^{A_{\infty}} \to \widehat{C(R_{\infty})}\), but we can say something when we modify the construction of \(\widehat{C(R_{\infty})}\) (see ((\ref{pCompFlatAndreFlatness}) in \Cref{BhattProof2})). So one question is that, if \(R\) is normal, then is the map \(R_{\perfd}^{A_{\infty}} \to \widehat{C(R_{\infty})}\) (almost) faithfully flat?
\end{remark}

Finally, we prove the remaining part of Main Theorem \ref{maintheorem}.

\begin{proof}[Proof of Main Theorem \ref{maintheorem} (4)]
In Main Theorem \ref{maintheorem} (4), we assume that $R$ is normal. That is, $R$ is integrally closed in the fraction field of $R$. By Cohen's structure theorem, there exists an unramified complete regular local ring $T  \defeq  W(k)[|t_2, \dots, t_d|]$ such that $T \hookrightarrow R$ is a module-finite extension. In particular, $T \hookrightarrow R$ is generically \'etale. Similar to the proof of \cite[Theorem 5.8]{shimomoto2018Integral}, there exists some $g \in T \setminus pT$ such that $T[1/pg] \hookrightarrow R[1/pg]$ is finite \'etale.

Note that $\widetilde{R}_{\infty}$ is a filtered colimit of $\widetilde{R}_{k}$ which is the integral closure of
$$
R_{k} = R[p^{1/p^k}, x_2^{1/p^k}, \dots, x_d^{1/p^k}]
$$
in $R_{k}[1/p]$. Set $r  \defeq  p x_2 \cdots x_n \in R$. Using the normality of $R$, similar to the proof of \cite[Lemma 5.7]{shimomoto2018Integral}, $R[1/r] \hookrightarrow R_{k}[1/r]=\widetilde{R}_{k}[1/r]$ is finite \'etale.

Applying \cite[Theorem 5.13]{gortz2010Algebraic} for the module-finite extension $T \hookrightarrow R$ and the associated surjective map $g : \Spec(R) \twoheadrightarrow \Spec(T)$, we have
\begin{equation*}
g(V_R(r)) = V_T(N_{L/K}(r)) \subseteq \Spec(T),
\end{equation*}
where $L  \defeq  \Frac(R)$ and $K  \defeq  \Frac(T)$ and $N_{L/K}$ is a norm map $N_{L/K} : L \to K$ satisfying $N_{L/K}(R) \subseteq A$. Then we have $V_R(r) \subseteq g^{-1}(V_T(N_{L/K}(r))) \subseteq \Spec(R)$ and taking complement shows that
\begin{equation*}
D_R(N_{L/K}(r)) = g^{-1}(D_T(N_{L/K}(r))) \subseteq D_R(r) \subseteq \Spec(R).
\end{equation*}
Then for $h \defeq pg N_{L/K}(r) \in T$, the composite map $T[1/h] \hookrightarrow R_{k}[1/h]$ is finite \'etale and this completes the proof.
\end{proof}

\appendix

\section{Witt-perfect Riemann's Extension Theorem of Perfect Rings}

In this section, we discuss that some results in \cite{nakazato2023Variant} (mainly in \S 5.1.2) can be generalized to the case of perfect rings, while \cite{nakazato2023Variant} considers only the \(p\)-torison-free case. The purpose is to complete the proof of \Cref{completionintegral} and so we mention only the minimum necessary results.

We use the following notation throughout this section.

\begin{notation}
Let \(V\) be a perfect valuation ring with a uniformizer \(t\) of characteristic \(p>0\) and let \(A\) be a \(t\)-torsion-free perfect \(V\)-algebra with a non-zero-divisor \(x \in A\). Assume that \(t, x\) is a regular sequence on \(A\). The symbol \(\widehat{(-)}\) is the \(t\)-adic completion.

Set \(A^j\) to be the Tate ring associated to \((A[t^j/x], (t))\) and let \(\mcalA^j\) be the \(t\)-completion of the Tate ring associated to \((\widehat{A}[t^j/x], (t))\) (see \cite[Definition 2.12]{nakazato2022Finite} for this terminology).
\end{notation}

\begin{lemma}[{\cite[Lemma 5.14 and Lemma 5.15]{nakazato2023Variant}}] \label{AppIsomAj}
    The following assertions hold for every \(j > 0\).
    \begin{enumerate}
        \item \(\widehat{A^j}\) and \(\mcalA^j\) are perfect \(\setF_p\)-algebras. In particular, they are uniform.
\item
The natural \(A\)-algebra map
\begin{equation} \label{AlmostIsomApp}
A[(t^j/x)^{1/p^\infty}]^{\wedge} \to \widehat{A}[(t^j/x)^{1/p^\infty}]^{\wedge}
\end{equation}
is an isomorphism.
\item
The localization of (\ref{AlmostIsomApp}) at \(t\) induces an isomorphism of Tate rings
        \begin{equation} \label{IsomTateRingAj}
            \widehat{A^j} \xrightarrow{\cong} \mcalA^j.
        \end{equation}
        \item The isomorphism (\ref{IsomTateRingAj}) induces an \(A[t^j/x]\)-algebra isomorphism
        \begin{equation} \label{IsomTateRingAjcirc}
            \widehat{A^{j\circ}} \xrightarrow{\cong} \mcalA^{j\circ}.
        \end{equation}
    \end{enumerate}
\end{lemma}

\begin{proof}
    (1): Since \(A\) is perfect, its localization and completion are perfect. Then $\widehat{A^j}$ and \(\mcalA^j\) are perfect and, in particular, uniform by \cite[Lemma 7.1.6]{bhatt2017Lecture}.

    (2): Since \(t, x\) is a regular sequence on \(A\), \cite[Proposition 3.14]{nakazato2023Variant} tells us that the natural map \(A[(t^j/x)^{1/p^n}]^{\wedge} \to \widehat{A}[(t^j/x)^{1/p^n}]^{\wedge}\) is an isomorphism for every \(n \geq 0\).
    Taking colimit of these isomorphisms and taking \(t\)-adic completion, we get the desired isomorphism (\ref{AlmostIsomApp}).

    (3): If the source (resp., target) of the isomorphism (\ref{AlmostIsomApp}) is a ring of definition of the Tate ring \(\widehat{A^j}\) (resp., \(\mcalA^j\)), then inverting \(t\) gives the desired isomorphism of Tate rings (\ref{IsomTateRingAj}). So it is enough to show that the source and target of (\ref{AlmostIsomApp}) are rings of definition of \(\widehat{A^j}\) and \(\mcalA^j\), respectively. Considering the inclusion \(tx A[(t^j/x)^{1/p^\infty}] \subseteq A[t^j/x]\), we have
\begin{equation*}
t^{j+1} A[(t^j/x)^{1/p^\infty}] \subseteq A[t^j/x] \subseteq A[(t^j/x)^{1/p^\infty}] \subseteq A^j
\end{equation*}
because of \(t^{j+1} = tx \cdot t^j/x\). Therefore, \(A[(t^j/x)^{1/p^\infty}]\) is a ring of definition of \(A^j\) and thus its \(t\)-adic completion \(A[(t^j/x)^{1/p^\infty}]^{\wedge}\) is a ring of definition of \(\widehat{A^j}\). Similarly, replacing \(A\) by \(\widehat{A}\) in the above argument, we can show that \(\widehat{A}[(t^j/x)^{1/p^\infty}]^{\wedge}\) is a ring of definition of \(\mcalA^j\).

    (4): By \cite[Proposition 2.4 (1)(b) and Lemma 2.13]{nakazato2022Finite}, the inclusion \(A[t^j/x] \hookrightarrow A^{j\circ}\) induces an isomorphism \((\widehat{A^j})^{\circ} \xrightarrow{\cong} \widehat{A^{j\circ}}\) of topological rings.
    On the other hand, the isomorphism (\ref{IsomTateRingAj}) induces an isomorphism \((\widehat{A^j})^{\circ} \xrightarrow{\cong} \mcalA^{j\circ}\) of topological rings. This completes the proof.
    % In the proof of \cite[Lemma 5.14 (2), Proposition 5.15 (4) and Theroem 5.16 (b)]{nakazato2023Variant}, they use the mixed characteristic assumption only in applying the Riemann's extension theorem (\cite[Theorem 8.1]{nakazato2023Variant}) but the theorem can be applied to the perfect \(\setF_p((T^{1/p^\infty}))\)-algebra \(\mcalA\) where \(T\) maps to \(t\). So the proof of \cite[Theorem 5.16 (b)]{nakazato2023Variant} can be applied to our case.
\end{proof}

Also, we can prove the following result which is a consequence of the perfectoid Riemann's extension theorem.

\begin{lemma}[{\cite[Theorem 5.16]{nakazato2023Variant}}] \label{AppWitt-perfectRiemannExt}
    There is an injective \(A\)-algebra map:
    \begin{equation} \label{CompletionLimit}
        \widehat{\lim_{j > 0} A^{j\circ}} \hookrightarrow \lim_{j > 0} \widehat{A^{j\circ}}
    \end{equation}
    which is \((x)^{1/p^\infty}\)-almost surjective.
\end{lemma}

\begin{proof}
    By the definition of \(t\)-adic completion and commutativity of limits, we have isomorphisms of \(A\)-algebras:
    \begin{equation*}
        \widehat{\lim_{j > 0} A^{j\circ}} \cong \lim_{n > 0}((\lim_{j > 0} A^{j\circ})/(t^n))~\text{and}~\lim_{j > 0}\widehat{A^{j\circ}} \cong \lim_{n > 0}\lim_{j > 0} (A^{j\circ}/(t^n)).
    \end{equation*}
    Applying the left exact functor \(\lim_{n > 0}\) for the exact sequence \(0 \to \{A^{j\circ}\}_{j > 0} \xrightarrow{\cdot t^n} \{A^{j\circ}\} \to \{A^{j\circ}/(t^n)\}_{j > 0} \to 0\), we have an injective map \((\lim_{j > 0} A^{j\circ})/(t^n) \hookrightarrow \lim_{j > 0}(A^{j\circ}/(t^n))\) and this makes an injective map of \(A\)-algebras
    \begin{equation*}
        \widehat{\lim_{j > 0} A^{j\circ}} \cong \lim_{n > 0}((\lim_{j > 0} A^{j\circ})/(t^n)) \hookrightarrow \lim_{n > 0}\lim_{j > 0} (A^{j\circ}/(t^n)) \cong \lim_{j > 0} \widehat{A^{j\circ}}.
    \end{equation*}
    We want to prove that this map is \((x)^{1/p^\infty}\)-almost surjective.
    The universality of the \(t\)-adic completion and limits, the \(A\)-algebra map \(\widehat{A} \to \widehat{\lim_{j > 0}A^{j\circ}} \hookrightarrow \lim_{j > 0}\widehat{A^{j\circ}}\) is the unique extension of the structure map \(A \to \lim_{j > 0} \widehat{A^{j\circ}}\) along \(A \to \widehat{A} \to \widehat{A^{j\circ}}\).
    By the perfectoid Riemann extension theorem (\cite[Theorem 8.1]{nakazato2023Variant} or \cite[Theorem 4.2]{bhatt2018Direct}), the canonical \(A\)-algebra map
    \begin{equation*}
        \widehat{A} \to \lim_{j > 0} \widehat{A^{j\circ}}
    \end{equation*}
    is \((x)^{1/p^\infty}\)-almost surjective (see also \cite[Lemma 3.5]{bhatt2018Direct}).
    As explained above, this map is the same as the map \(\widehat{A} \to \widehat{\lim_{j > 0}A^{j\circ}} \hookrightarrow \lim_{j > 0}\widehat{A^{j\circ}}\). The latter map becomes \((x)^{1/p^\infty}\)-almost surjective and this completes the proof.
\end{proof}

Next we consider the invariant ring of \(G\)-Galois covering after taking \((-)^{j\circ}\).

\begin{lemma} \label{GaloisStableAj}
    Let \(A[1/tx] \hookrightarrow C[1/tx]\) be a \(G\)-Galois extension for some finite group \(G\) and \(A\)-algebra \(C\) such that \(t\) and \(x\) are non-zero-divisors of \(C\).
    Assume that \(C\) is the integral closure of \(A\) in \(C[1/tx]\). Then there exists an isomorphism
    \begin{equation}
        \widehat{A^{j\circ}} \xrightarrow{\cong} \parenlr{\widehat{C^{j\circ}}}^G
    \end{equation}
    for every \(j > 0\), where \(C^j\) is the Tate ring associated to \((C[t^j/x], (t))\) as \(A^j\).
\end{lemma}

\begin{proof}
    Let us prove that \(A^j\) is preuniform in the sense of \cite[Definition 2.14]{nakazato2022Finite}. First, recall that \(t^{j+1} A[(t^j/x)^{1/p^\infty}] \subseteq A[t^j/x]\) and \(A[(t^j/x)^{1/p^\infty}]\) is a perfect $\mathbb{F}_p$-algebra. As discussed in the proof of \cite[Lemma 7.1.6]{bhatt2017Lecture}, we get that \(A[(t^j/x)^{1/p^\infty}]\) is \((t)^{1/p^\infty}\)-almost isomorphic to \(A^{j\circ}\), which proves that $A^j$ is preuniform. Hence the hypotheses of \cite[Proposition 4.7 and Proposition 4.8]{nakazato2023Variant} are satisfied by setting  $A_0 = A$, $g = x$, $B_0 = C$ and $B' = C[1/tx]$. By \cite[Proposition 4.8 (3)]{nakazato2023Variant}, we have an isomorphism of topological rings
    \begin{equation*}
        \mathcal{C}^j \cong C[1/tx] \otimes_{A[1/tx]} \mcalA^j
    \end{equation*}
    and in particular, the structure of a Tate ring $\mathcal{C}^j$ is compatible with the canonical structure of a Tate ring over \(C[1/tx] \otimes_{A[1/tx]} \mcalA^j\) as in \cite[Lemma 4.1]{nakazato2023Variant}. As $A[1/tx] \hookrightarrow C[1/tx]$ is a $G$-Galois covering, its base change $\mcalA^j \hookrightarrow \mathcal{C}^j$ is also a $G$-Galois covering by \cite[Lemma 12.2.7]{ford2017Separable}.
    Since $\mcalA^j$ is uniform by in \Cref{AppIsomAj} (1), we see that \(\mcalA^{j \circ} = (\mathcal{C}^{j \circ})^G\) by \cite[Lemma 4.5]{nakazato2023Variant}.
    The above (\ref{IsomTateRingAjcirc}) in \Cref{AppIsomAj} shows that
    \begin{equation*}
        \widehat{A^{j \circ}} \xrightarrow{\cong} \mcalA^{j \circ}=\left(\mathcal{C}^{j \circ}\right)^G \xrightarrow{\cong} \left(\widehat{C^{j \circ}}\right)^G
    \end{equation*}
    and this completes the proof.
\end{proof}

Notice that there is another way to deduce \Cref{GaloisStableAj}. This is found in \cite[Discussion 5.22 (1)]{nakazato2023Variant} by using Faltings' almost purity theorem (see \cite[Theorem 2.4]{faltings1988PAdic} or \cite[Proposition 3.4]{olsson2009Faltings}).

\end{document}